\documentclass[12pt]{article}
\usepackage{amsmath,amsthm,amsfonts,amssymb,mathrsfs,bm}
\usepackage[usenames]{color}
\usepackage{hyperref}
\usepackage{amssymb}

\parskip        \bigskipamount

\newtheorem{theorem}{Theorem}[section]
\newtheorem{lemma}[theorem]{Lemma}

\newtheorem{proposition}[theorem]{Proposition}

\newtheorem{remark}{Remark}[section]

\newtheorem{definition}{Definition}

\def\R{\mathbb{R}}
\def\C{\mathbb{C}}

\def\E{\mathbb{E}}
\def\P{\mathbb{P}}

\def\el{\mathcal{L}}
\def\wt{\widetilde}
\def\p{\mathfrak{P}}
\def\ol{\overline}
\def\olz{\overline{z}}
\def\olp{\overline{\P}}
\def\eps{\epsilon}
\def\del{\delta}
\def\min{\mathrm{min}}
\def\M{\mathcal{M}}
\def\S{\Sigma_{\mathrm{a}}}
\def\L{\Lambda}
\def\n{\mathcal{N}}
\def\I{\mathcal{I}}
\def\t{\theta}
\def\p{\mathcal{P}}
\def\d{\mathcal{D}}

\def\G{\Sigma}
\def\wn{\wt{\mathcal{N}}}
\def\w{\mathcal{W}}

\def\LL{\mathbb{L}}
\def\supp{\mbox{\rm supp}}
\def\h{\mathrm{Har}}
\def\mm{\mathcal{M}_1}
\def\c{\complement}
\def\a{\alpha}
\def\D{\mathcal{D}}
\def\b{\mathcal{B}}

\renewcommand{\l}[0]{\left }
\renewcommand{\r}[0]{\right}
\renewcommand{\i}{\infty}

\makeatletter
\renewcommand*{\@cite@ofmt}{\hbox}
\makeatother

\begin{document}

\author{Subhroshekhar Ghosh\thanks{ORFE, Princeton University, 
Princeton, NJ 08544, USA. Email: subhroshekhar@gmail.com.}\and
Ofer Zeitouni\thanks{Faculty of Mathematics,
  Weizmann Institute of Science, POB 26, Rehovot 76100, Israel and 
  Courant Institute, 251 Mercer Street, New York, NY 10012, U.S.A.. Email: ofer.zeitouni@weizmann.ac.il.
Research partially supported by
a grant from the Israel Science Foundation.}}
\date{September 24, 2014. Revised March 13, 2015.}
 \title{Large deviations for zeros of random polynomials with i.i.d.  exponential coefficients} 
 \maketitle
\begin{center}
  Dedicated to the memory of Wenbo Li
   \end{center}
 \begin{abstract}We derive a large deviation principle for the empirical measure of zeros of the random polynomial $P_n(z)=\sum_{j=0}^n \xi_j z^j$, where the coefficients $\{\xi_j\}_{j\geq 0}$ form an i.i.d. sequence of exponential 
random variables.
 \end{abstract}
\section{Introduction}
\label{intro}
The study of the zero set $\{z_1,\ldots,z_n\}$ 
of random polynomials 
\begin{equation}
        \label{eq-poly}
        P_n(z)=\sum_{j=0}^n \xi_j z^j
\end{equation}
with i.i.d. coefficients $\{\xi_j\}_{j\geq 0}$ has a long and rich 
history, which we will not review here; see \cite{BA} for a classical account and \cite{TV} for the most recent results. 
Under mild conditions, the convergence of the empirical measure 
$L_n=\frac1n \sum_{i=1}^n \delta_{z_i}$ of zeros of $P_n$ 
to the uniform measure on the unit circle goes back at least to \cite{sparo-sur} and \cite{erdos-turan}; scaled version of this convergence can be found in \cite{shepp-vanderbei} (for the Gaussian case) and \cite{IbragimovZ} (for more general i.i.d. coefficients
in the domain of attraction of stable laws). 

We are interested in the large deviations for the 
empirical measure $L_n$.
In the case of Gaussian coefficients, 
this has been studied before \cite{ZelditchZ},
\cite{berman}, 
\cite{Bloom}, exploiting methods related to those used
in the study of random matrices from the classical $\beta$-ensembles
\cite{BAG,BAZ,Ser}. Like in the case of random matrices, when one ventures away from the Gaussian setup (with i.i.d. coefficients), 
not much is known concerning large deviations.

Our goal in this paper is to exhibit a new class of coefficients, for which a large deviation principle for the empirical measure can be proved, namely
the class of i.i.d. exponential coefficients, which for concreteness 
we normalize to have parameter $1$. To our knowledge, the first to consider 
explicitly asymptotics for this class was Wenbo Li \cite{Li}, who 
used general formulae of Zaporozhets \cite{Za} in order to compute  the  
probability that all roots in such a polynomial are real. 
We relate our result to Li's computation in Theorem \ref{theo-Lihom} below.
  
In order to state our results, we introduce some notation. 
In the rest of the paper, $P_n$ denotes a random polynomial as in
\eqref{eq-poly}, with i.i.d. exponential (of parameter $1$)
coefficients $\{\xi_i\}$
and associated empirical measure of zeros $L_n$.
For any Polish space $X$, let $\M_1(X)$ denote the space of probability
measures on $X$, equipped with the topology of weak convergence. Let $\mbox{\rm pol}_+$ denote the collection of polynomials 
(over $\C$) with coefficients that are real positive. 
For $p\in \mbox{\rm pol}_+$, let $\mu_p\in \M_1(\C)$ 
denote the 
empirical measure
of zeros of $p$. Note that $\mu_p$ depends on the set of 
zeros and not on a particular labeling of the zeros, that
$\mu_p$ is symmetric with respect to the transformation $z\mapsto z^*$, and that
$\mu_p(\R_+)=0$. (Here and in the sequel, we use $\R_+$ to denote the interval
$(0,\infty)$.)
Finally, for any space $X$ and subset $A\subset X$, we let 
$A^\complement$ denote the complement of $A$ in $X$.

We introduce the closure of the collection of empirical measures of polynomials with positive coefficients  
\begin{equation}
        \label{eq-calP}
        \p=\overline{ \{\mu_p: p\in \mbox{\rm pol}_+\}} \subset \M_1(\C)\,. 
\end{equation}
Obviously, $L_n\in \p$. A characterization of $\p$, due 
to Bergweiler and Eremenko, appears in Theorem \ref{Ere} below.

\begin{definition}
 \label{logen}
 For any measure $\mu \in \M_1(\C)$, define the logarithmic potential function to be \[\LL_{\mu}(z)=\int \log |z-w| d\mu(w)\] 
 and the logarithmic energy to be 
 \[\Sigma(\mu)= \iint \log|z-w| \mu(z) \mu(w).\]
\end{definition}

\begin{definition}
 \label{prerf}
 Define the function $I:\M_1(\C)\to \R_+$
 by  
 \[ I(\mu)= \left \{ \begin{array} {ll} \int \log|1-z|d\mu(z) - \frac12\iint \log |z-w| d\mu(z) d\mu(w), &\mbox{\rm  if }  \mu \in \p, \\ \infty, &\mbox{\rm if } \mu \notin \p \end{array} \right.  \] 
\end{definition}
We will see in Section \ref{goodrate} 
that $I$ is well defined (for $\mu\in \p$, as  the integral with respect 
to $\mu\times\mu$ of the function $f(z,w)=\log|1-z|+\log|1-w|-\log|z-w|$)
and non-negative (the latter fact is immediate from
the lower
bound in Lemma \ref{lem-LB})\footnote{A. Eremenko showed us a direct proof of the non-negativity of $I$, that bypasses the use of the lower bound from Lemma \ref{lem-LB}. Since we need the latter lemma for other reasons, we do not reproduce his proof here.}. 

Our main result concerning large deviations of $L_n$ is the following.
\begin{theorem}
\label{ldp}
The random measures $L_n$ satisfy a large deviation principle in the space $\M_1(\C)$ with speed $n^2$ and  good rate function $I$. Explicitly, we have:
\begin{itemize}
        \item[(i)] The function 
                $I:\M_1(\C)\to [0,\infty]$ has compact level sets, 
                i.e. the sets  $ \{ \mu: I(\mu) \le M \}$ are compact subsets of  $\M_1(\C)$ for each  $M \in \R$. 
 \item[(ii)] For each open set $O \subset \M_1(\C)$, we have \[\liminf_{n \to \infty} \frac{1}{n^2} \log \P_n ( L_n \in O) \ge - \inf_{\mu \in O}I(\mu).\]
 \item[(iii)] For each closed set $F \subset \M_1(\C)$, we have \[\limsup_{n \to \infty} \frac{1}{n^2} \log \P_n ( L_n \in F) \le - \inf_{\mu \in F}I(\mu).\]
\end{itemize}
\end{theorem}
Comparing the statement of Theorem \ref{ldp} with the main results in \cite{BAZ} and \cite{ZelditchZ}, one sees that in spite of the fact that we are dealing with zeros of random polynomials, the rate function is closer to a random matrix theory rate function than to the one appearing in the Gaussian case. This is due to the expression for the joint distribution of zeros, see Section
\ref{sec-prelim-I} below.  
We also note that because $I$ is a good rate function, 
any minimizer $\mu$ of $I(\cdot)$ in $\M_1(\C)$ must satisfy that
$I(\mu)=0$; in particular, the uniform measure on the unit circle is 
a minimizer, 
 as one expects from the limit results in
\cite{erdos-turan}, \cite{sparo-sur}. 
The strict convexity of $I$
(which follows from the same argument as in  \cite{BAG} together
with the convexity of $\p$) shows that
it is the unique minimizer.

As mentioned above, we tie our results to Li's computation in \cite{Li}. Toward this end, let $\R_-=\R\setminus\R_+$ and define
$\mu_R\in \M_1(\R_-)\subset \M_1(\C)$ to be such that 
$I(\mu_R)=\inf_{\mu\in \M_1(\R_-)}I(\mu)=:I_R$. 
(Such a minimizer exists due to the lower semicontinuity of $I$.)
\begin{theorem}
\label{theo-Lihom}
Conditioned on $L_n\in \M_1(\R_-)$, 
the sequence of random empirical measures $L_n$ satisfy the 
large deviation principle in $\M_1(\R_-)$ with speed $n^2$ and rate function $I_R(\mu)=I(\mu)-I_R$. In particular, conditioned on $L_n\in M_1(\R)$, the sequence $L_n$ converges weakly to $\mu_R\in M_1(\R_-)$.
\end{theorem}
A characterization of $\mu_R$ is given in the next theorem, due to J. Baik.
\begin{theorem}
  \label{theo-baik}
  The minimizer $\mu_R$ has density with respect of
  Lebesgue measure on $\R_-$ equal to 
  \begin{equation}
    \label{eq-baik}
    \phi(x)=\frac{1}{\pi (|x|+1)\sqrt{|x|}}{\bf 1}_{\{x<0\}}\,.
  \end{equation}
\end{theorem}
\noindent
An interesting feature of the minimizer $\mu_R$ is that 
it is not compactly supported.
We discuss Theorems \ref{theo-Lihom} and \ref{theo-baik} in Section \ref{sec-6}.

\vskip 1cm
\noindent
{\bf History and Acknowledgements:} 
Our interest in this problem started when one of us
(O.Z.) attended a talk by Wenbo Li on \cite{Li}; that talk
suggested that an underlying 
large deviation principle should exist in the real case, and 
J. Baik computed its equilibrium measure, 
repeated here as Theorem \ref{theo-baik}.  
Wenbo Li's untimely  death
prompted S. G. and O. Z.  
to revisit the problem, and the important
role of the class $\p$ 
in the complex case emerged. 
We posted the question concerning the characterization of $\p$
on MathOverflow \cite{MO}, and the question was answered in \cite{BES}.

We are indebted to J. Baik for allowing us to use his proof of Theorem 
\ref{theo-baik}, and to
A. Eremenko for making \cite{BES} available to us as a preprint,  for
his patience in answering our questions, and for his comments on a preliminary
draft of this paper. We thank the anonymous referee 
for a careful reading of the paper and for spotting 
an error in the original version of this article.

\section{Preliminaries}
We discuss in this section several preliminaries. We first introduce the joint distribution of zeros and  then we describe properties of $\p$.
\subsection{The joint distribution of zeros}
\label{sec-prelim-I}
Let $p\in \mbox{\rm pol}_+$ be of degree $n$ with $n-2k$ real zeros, 
$k=0,1,\ldots,\lfloor n/2\rfloor$. We consider the zeros of 
$p$ as a vector $(z_1,\cdots,z_n)$ with the convention that $z_1,\cdots, 
z_k$ are the 
non-real zeros with positive imaginary part, 
$z_{k+1}=\olz_1,\cdots,z_{2k}=\olz_{k}$ and  $z_{2k+1},\cdots,z_{n}$ 
denote the $n-2k$ real zeros. In this notation, for $k$ fixed, a
set of zeros is generically  mapped to $k! (n-2k)!$ distinct points in 
$A_{n,k}^+=\C_+^k\times \C_-^k\times \R^{n-2k}$, and 
$A^+_{n,k}$ is parametrized by $\C_+^k\times \R^{n-2k}$.

Performing the change of variables from 
$(\xi_0,\ldots,\xi_n)$ to $(z_1,\ldots,z_k,z_{n-2k+1},\ldots,z_n,\xi_n)$,
counting multiplicities, using the form of the exponential density
and integrating over the density of $\xi_n$
(see \cite{Za} for a similar computation), one has that the random 
polynomial $P_n$ induces the following measure on $\C^n$:
\begin{eqnarray}
 \label{dist1}
 &&d\P_n(z_1,\cdots,z_n)=\\
 &&\sum_{k=0}^{\lfloor n/2 \rfloor} \frac{2^{k}}
 {k!(n-2k)!}\frac{\prod_{1\le i < j \le n} 
 |z_i - z_j|}{\prod_{j=1}^n |1-z_j|^{n+1}} 
 1_{B_{n,k}}(z_1,\cdots,z_n) 
 d\el(z_1)\cdots 
 d\el(z_k) d\ell(z_{2k+1})\cdots d\ell(z_n).\nonumber
\end{eqnarray}
Here $\el$ is the Lebesgue measure on $\C$,
$\ell$ is the Lebesgue measure 
on $\R$, and $B_{n,k}$ consists of the $n$-tuples
$(z_1,\ldots,z_n)\subset A_{n,k}$ that can be obtained as the zero set (with $n-2k$ real zeros)
of a polynomial
of degree $n$ with positive coefficients.
In particular, letting $A_{n,k}=(\C\setminus\R)^{2k}\times \R^{n-2k}
\subset \C^n$, we see that the density of 
$\P_n$ on any fixed $A_{n,k}$ 
is
\begin{equation}
 \label{dist2}
 \frac{1}{\mathcal{Z}_{n,k}}
 1_{B_{n,k}}(z_1,\cdots,z_n) 
 \exp\l( \sum_{1\le i < j \le n} \mathrm{log} |z_i - z_j| - 
 (n+1) \sum_{j=1}^n \mathrm{log}|1-z_j| \r).
\end{equation}
where the constants $\mathcal{Z}_{n,k}$ satisfy that
\begin{equation}
        \label{eq-znk}
        \lim_{n\to\infty}
        \frac{1}{n^2}\log
        \max_{k=1}^{\lfloor n/2\rfloor}
        \mathcal{Z}_{n,k}=
        \lim_{n\to\infty}
        \frac{1}{n^2}\log \,
        \min_{k=1}^{\lfloor n/2\rfloor}
        \mathcal{Z}_{n,k}=
        0\,.
\end{equation}
The representation \eqref{dist2} with \eqref{eq-znk}
is particularly suited for LDP analysis.

\subsection{Properties of the class $\p$ of measures}
Obviously, for any $p\in \mbox{\rm pol}_+$ with $\mu_p$ its 
empirical measure of zeros, we have that $\mu_p(\R_+)=0$. However, that 
property is not preserved by weak convergence, and hence 
a-priori it is not clear that all measures in $\p$ satisfy it
(although we will see, as a consequence of Obrechkoff's theorem below,
that in fact they do). 
In this subsection, we discuss this and other properties of the class $\p$. 

\subsubsection{Obrechkoff's Theorem}
A starting point for
the description of $\p$ is the following classical theorem.
\begin{theorem}[Obrechkoff]
 \label{obr1}
 Let $p\in \mbox{\rm pol}_+$ and let
 $$C_\alpha=\{z\in \C: |\arg{z}|\leq \alpha\}$$
denote
the symmetric (around the positive real line)
cone in $\C$ with apex at the origin and angle $2\alpha$.
Then, $\mu_p(C_\alpha)\leq 2\alpha/\pi$.
\end{theorem}
The proof, given in \cite{obrechkoff}, uses the argument principle. 
For our needs, note that Obrechkoff's Theorem implies that
$\mu(C_\alpha)\leq 2\alpha/\pi$ for any $\mu\in \p$. In particular,
$\mu(\R_+)=0$ for such $\mu$.

Obrechkoff's Theorem
leads to the following lemma on the integrability of the logarithm near 1 for 
$\mu\in \p$.
 
\begin{lemma}
\label{obr2}
Let $M>0$ and  set
$A_M=\{z:\log|1-z| \le -M\}$. 
Then there is a positive quantity $C(M)$ satisfying
$\lim_{M\to\infty} C(M)=0$ such that for any $\mu\in\p$,
\begin{equation}
\label{eq-lem2.2}
\max\{\mu(A_M),\int_{A_M} \l| \log|1-z| \r| d\mu(z)\}\le C(M).
\end{equation}
\end{lemma}

\begin{proof}
        We first consider $p\in \mbox{\rm pol}_+$ of degree $N$. For $M>0$, let
        $Z_M$ consist of all zeros $z_i$ of $p$ such that $\log|1-z_i|\le -M$. 
        Let $N(p,M)$ be the cardinality of $Z_M$ and $S(p,M)=\sum_{i:z_i\in Z_M}
        \log|1-z_i|$. By Theorem 
        \ref{obr1}, there exists a constant $M_0$ independent of $p$ or $N$ 
        such that  for $M>M_0$,
$N(p,M) \le 4e^{-M}N$. Thus, 
with $B_j:=\{z: -(j+1)<\log|1-z| \le -j: j \ge M\}$, 
we get
\[\frac{1}{N}|S(p,M)| \le 
        \frac1N \sum_{j=M}^\infty \sum_{i: z_i\in B_j}|\log|1-z_i||
        \leq
\sum_{j=M}^{\infty} j \cdot 4e^{-j} =:c(M)\,,\] 
with $c(M)\to_{M\to\infty} 0$. 
Since $|S(p,M)|\geq N(p,M)$ if $M_0$ is chosen large enough, we obtain
the same inequality for $N(p,M)/N$. Thus, 
\eqref{eq-lem2.2} holds for $\mu_p$, uniformly in $p,N$.

To obtain the same inequality for $\mu\in\p$, take an approximating sequence
$\mu_{p_n}\to\mu$, and use that 
$\mu(A_M)\leq \limsup_{n\to\infty} \mu_{p_n}(A_{M-1})$ together with
$$\int_{A_M} |\log|1-z|\vee -K|d\mu(z)\leq 
\limsup_{n\to\infty}
\int_{A_{M-1}} |\log|1-z|\vee -K|d\mu_{p_n}(z)\leq 
c(M-1)\,,$$
and then apply monotone convergence over $K$.
One concludes that \eqref{eq-lem2.2} holds with $C(M)=c(M-1)$.
\end{proof}

\subsubsection{The Bergweiler-Eremenko Theorem}
For $\mu \in \M_1(\C)$,
let
$$ \widehat{\LL}_\mu(z)=\int_{|w|\leq 1} \log(|z-w|)d\mu(w)+
\int_{|w|>1} \log(|1-\frac{z}{w}|)d\mu(w)\,.$$
Whenever $C_\mu:=\int \log_+|w|\mu(dw)<\infty$, it holds that 
$$\widehat{\LL}_{\mu}(z)=\LL_{\mu}(z)-C_\mu\,.$$ 
In a recent work \cite{BES}, Bergweiler and Eremenko  proved the following.
\begin{theorem}[Bergweiler-Eremenko]
\label{Ere}
$\mu\in \p$ if and only if
it is invariant with respect to conjugation and satisfies
$\widehat{\LL}_{\mu}(z) \le \widehat{\LL}_{\mu}(|z|)$ for all $z \in \C$.
\end{theorem}
In the proof of Theorem \ref{ldp},
we will exploit this result, and in addition, its proof. 

\subsection{Plan of the proof of Theorem \ref{ldp}}
The proof of Theorem \ref{ldp} is divided into sections.
Section \ref{rateprop}, 
is devoted to establishing  
properties of $I$ and to a proof of the upper bound in Theorem \ref{ldp}.
Sections \ref{potential} and \ref{meastop} deal with 
potential theoretic preliminaries 
that play an important role in the proof of the lower bound 
in Theorem \ref{ldp}.
Section \ref{lbound} states and proves Lemma \ref{lem-LB}, 
which is the lower bound; the proof of Lemma \ref{lem-LB} uses some
some technical
approximation lemmas whose proofs are  
postponed to Sections \ref{sec-proofstage1-1} and \ref{erapprox}.

\section{Properties of the rate function}
\label{rateprop}
In Section \ref{goodrate}),
we establish 
properties of $I$ and establish that it is well defined and 
lower-semicontinuous.
In Section \ref{etight}, we prove the exponential tightness of $\{\P_n\}$. 
Section \ref{ubound} is devoted to the proof of the upper bound. 

\subsection{$I$ is well defined and has compact level sets.}
\label{goodrate}


Here we prove the following.
\begin{lemma}
        \label{lem-Ilsc}
        The function $I$ is well defined on $\M_1(\C)$ and it possesses 
        compact level sets.
\end{lemma}
Lemma \ref{lem-Ilsc} almost shows that $I$ is a good rate function; what 
is missing is a proof that $I(\mu)\geq 0$ for $\mu\in \M_1(\C)$. This fact
is a consequence of Lemma \ref{lem-LB} below.
\begin{proof}
Define $f(z,w)=\log|1-z|+\log|1-w|-\log|z-w|$.
We first show that one can choose a function $K(L)\to_{L\to\infty} \infty$ so 
that the following inclusion holds for all $L$ large:
\begin{equation} \label{gr1}\{(z,w): |z|>L,|w|>L\}
        \subset \{ f(z,w) \ge K(L) \}. \end{equation}
Indeed,
setting $z'=1-z$ and $w'=1-w$, we get 
\[f(z,w)=\log|z'|+\log|w'| - \log|z'-w'|.\] 
But 
\begin{equation}
	\label{eq-prague}
	\frac{|z'w'|}{|z'-w'|}\ge \frac{1}{\frac{1}{|z'|}+\frac{1}{|w'|}} 
\ge \frac{1}{2} \min \{ |z'|,|w'| \}. 
\end{equation}
Clearly, this implies (\ref{gr1}).  
Further, the last inequality also implies that, with
$A=\{(z,w)\in \C^2: |1-z|>1/4, |1-w|>1/4\}^\complement$,
we have
\begin{equation}
        \label{eq-LBa}
        \inf_{A^\complement} f(z,w) \geq -\log 8\,.
\end{equation}

We next show that $I(\mu)$ is well defined. For that it is
enough to consider $\mu\in \p$. Since $f(z,w)\geq c+\log \min(|z-1|,|w-1|)$ for some constant $c$, an application of 
Lemma 
\ref{obr2} implies that the integral of $f$ is well defined (and bounded below).

We next show that the level sets of $I$ are precompact. Choose $L$ 
large enough so
that $K(L)>1$.
Then,
\begin{eqnarray}
        \label{eq-multipg6}
\mu(|z|>L)^2&= &
\mu \otimes \mu (|z|>L,|w|>L)\\
&
\le& \mu \l( \{f(z,w)\ge K(L) \} \cap \{|1-z|>1/4,|1-w|>1/4\}\r)
\nonumber\\
&\le&  \frac{1}{K(L)+\log 8} \iint_{A^\complement} 
(f(z,w)+\log 8) d\mu(z)d\mu(w)
\nonumber\\
&\le & 
\frac{1}{K(L)+\log 8} \l(\l(\iint f(z,w) d\mu(z) d\mu(w) - 
\iint_A f(z,w) d\mu(z) d\mu(w)  \r)+\log 8\r)\,,
\nonumber
\end{eqnarray}
where we used \eqref{eq-LBa} in the second inequality.

Our next task is to show that 
\begin{equation}
\label{eq-slinepg6}
-\iint_A f(z,w) d\mu(z) d\mu(w) \le c\,,
\end{equation}
for some constant $c$ independent of $\mu\in \p$. 
To this end, we write $A=A_1\cup A_2\cup A_3\cup A_4$ with
\begin{eqnarray*}
        &&
        A_1:=\{|1-z|\le 1/2, |1-w|\le 1/4\},\; 
        A_2:=\{|1-z|>1/2,|1-w|\le 1/4\},\\
        &&A_3:=\{|1-w|\in [1/4,1/2], |1-z|\leq 1/4\},\;
        A_4:=\{|1-w|>1/2, |1-z|\leq 1/4\} .
\end{eqnarray*}
Since $|z-w|\leq 3/4$ for $(z,w)\in A_1$, we have
\begin{eqnarray}
        \label{eq-A1}
        -\iint_{A_1}f(z,w)d\mu(z)d\mu(w)&\leq& 
\iint_{A_1} \log |z-w| d\mu(z)d\mu(w) - 2
\int_{\{z: |1-z|\leq 1/2\}} 
\log|1-z|d\mu(z)\nonumber\\
&\leq & 
- 2
\int_{\{z: |1-z|\leq 1/2\}} 
\log|1-z|d\mu(z)\leq 2C(\log 2)\,,
\end{eqnarray}
where $C(\log 2)$ is given by Lemma \ref{obr2}.
With the same argument, we also have 
\begin{equation}
\label{eq-A3}
-\iint_{A_3}f(z,w)d\mu(z)d\mu(w)\leq  2C(\log 2).
\end{equation}
%

For the integral over the set $A_2$, we note that $|1-(1-w)/(1-z)|\in (1/2,3/2)$
for $(z,w)\in A_2$, and therefore
\begin{eqnarray}
        \label{eq-A2}
&&      -\iint_{A_2}f(z,w)d\mu(z)d\mu(w)=-\iint_{A_2} 
\log\frac{|1-w|}{|1-\frac{1-w}{1-z}|}d\mu(z)d\mu(w)\\
&\leq&
\log(3/2)-\int_{\{w:|1-w|\leq 1/4\}} \log|1-w|d\mu(w)\leq
\log(3/2)+C(\log 4)\,,
\nonumber
\end{eqnarray}
where
$C(\log 4)$ is again given by Lemma \ref{obr2}. 

Since $\iint_{A_4}f(z,w)d\mu(z)d\mu(w)=
\iint_{A_2}f(z,w)d\mu(z)d\mu(w)$, 
we obtain by combining \eqref{eq-A1}, \eqref{eq-A3}  and \eqref{eq-A2}
that
\eqref{eq-slinepg6} holds.

From \eqref{eq-multipg6} and  \eqref{eq-slinepg6} we obtain that for any $M>0$,
$$\sup_{\{\mu: I(\mu)\leq M\}} \mu(|z|>L)\to_{L\to\infty} 0\,,$$
which yields the pre-compactness of the level sets of $I$ 
by an application of Prohorov's criterion.


It remains  to show that $I$ is lower semicontinuous. Since $\p$ is closed
in $\M_1(\C)$, it is enough to check the lower
semicontinuity in $\p$. 
%
Toward this end, for $\epsilon,M>0$ define
\[f^{\eps,M}(z,w)= \l[ \l(\log|1-z| \vee (-\frac{1}{\eps}) \r) + \l(\log|1-w| \vee (-\frac{1}{\eps})\r) -\l( \log|z-w| \vee (-M) \r) \r] \wedge M\]
and
\[f^{\eps}(z,w)=\log|1-z| \vee (-\frac{1}{\eps}) +\log|1-w| \vee (-\frac{1}{\eps}) -\log|z-w|.\]
Set \[I^{\eps,M}:= \frac{1}{2} \iint f^{\eps,M}(z,w) d\mu(z) d\mu(w) \] 
and \[I^{\eps}:= \frac{1}{2}\iint f^{\eps}(z,w) d\mu(z) d\mu(w).\] 
Note that by monotone convergence,
$I^{\epsilon}=\sup_{M>0} I^{\epsilon,M}$, and since $I^{\epsilon,M}:\M_1(\C)
\to \R$ is continuous, we have that $I^\epsilon$ is lower semicontinuous on
$\M_1(\C)$, and therefore on $\p$. On the other hand, 
$I^\epsilon$ converges uniformly to $I$ on $\p$ by Lemma \ref{obr2}. 
It follows that $I$ is also lower semicontinuous on $\p$, completing the proof
of the lemma.
\end{proof}

\subsection{Exponential Tightness of $\{\P_n\}$.}
\label{etight}
We prove in this subsection the exponential tightness of the family $\{\P_n\}$.
\begin{lemma}
        \label{exp-tight}
        The family 
$\{\P_n\}$ is exponentially tight. That is, with $T>0$ 
there exist compact sets $K_T\subset
\p$ so that
$$\limsup_{n\to\infty}\frac{1}{n^2}\log P(L_n\in K_T^\complement)
\leq -T.$$
\end{lemma}
\begin{proof}
        Introduce the function
        $g(z,w)=\log |1-z|+\log |1-w|-\log_+(|z-w|)$,
        and define the function $J$ on $\p$ by
        $$J(\mu)=\iint g(z,w)d\mu(z)d\mu(w)\,.$$
        We first claim that the sets
	$$ K_B:=\{\mu\in\p: J(\mu)\leq 5B\}$$
        are compact in $\M_1(\C)$ for $B$ large.
	The proof is very similar to the argument in 
	subsection \ref{goodrate}, and therefore we only sketch it.
	First, using \eqref{eq-prague} if $|z-w|>1$ and a direct computation
	otherwise, one gets that 
	\eqref{gr1} remains true (with a different choice of 
	$K(L)\to_{L\to\infty}\infty$) if $f$ is replaced by $g$.
	One also gets \eqref{eq-LBa} for $g$ and the same $A$ 
	with $8$ replaced by $16$. One then easily sees,
	using Lemma \ref{obr2},
	that $J(\mu)$ is 
	well defined. Next, arguing as in \eqref{eq-multipg6} with 
	the same substitutions of $g$ for $f$ and change of constants,
	one reduces the proof of pre-compactness of the level sets of
	$J$ to the proof of \eqref{eq-slinepg6}; for the latter,
	one splits the integral over  the sets $A_i$, $i=1,\ldots,4$,
	and notes that the argument given for $A_1$ and $A_3$ still applies
	when replacing $f$ by $g$, while on $A_2$ and $A_4$ one has that 
	$z-w\geq 1/2$ and therefore $f(z,w)\leq g(z,w)+C$ for an appropriate 
	constant $C$. This allows one to repeat the argument given for $f$ and conclude \eqref{eq-slinepg6} for $g$. Finally, the lower semicontinuity
	for $J$ follows in the same way as for $I$. This completes the
	proof of compactness of the sets $K_B$.

We need thus to estimate $P(L_n\in K_B^\complement)$.
Introduce the random variables
\[X_n= \frac{1}{n^2}\sum_{i=1}^{n} \log|1-z_i|=
\frac{1}{n^2}\log \frac{P_n(1)}{\xi_n}=\frac{1}{n^2} \log \frac{\xi_0+
\cdots+\xi_n}{\xi_n}\] 
and
\[Y_n=\frac{1}{n}\sum_{\{i:|1-z_i|<1\}} |\log|1-z_i||\,.\]
We need the following estimate, whose proof is postponed to the end of 
the subsection.
\begin{lemma}
        \label{lem-xn}
        There exists a constant $c>0$ such that
        for all $n$ large,
        \begin{equation}
                \label{eq-Xnnew}
                \P_n(|X_n|>B)\leq 20ne^{-Bn^2}
        \end{equation}
        and
        \begin{equation}
                \label{eq-Ynnew}
                \P_n(Y_n>c)=0\,.
        \end{equation}
\end{lemma}

Continuing with the proof of Lemma \ref{exp-tight},
we have
\begin{eqnarray}
        \label{eq-111413b}
        \nonumber
        \P_n(L_n \in K_B^\complement) &\le& 
        \P_n \l(\{L_n \in K_B^\complement\} 
\cap \{|X_n|\le B\}\r) + 
\P_n \l( \{|X_n|> B\}\r)\\
&\leq& \sum_{k=0}^{\lfloor n/2 \rfloor}
        \P_n \l(\{L_n \in K_B^\complement\} 
        \cap \{|X_n|\le B\}\cap A_{n,k}\r) + 20ne^{-Bn^2}\,, 
\end{eqnarray}
see \eqref{dist2} for the definition of $A_{n,k}$.

We next consider the density of $\P_n$ on $A_{n,k}$, see \eqref{dist2}, 
which we
write as
\begin{eqnarray*}
         &&f_{k,n}(z_1,\cdots,z_N)=\\
         &&
\frac{1}{{\mathcal Z}_{n,k}}
\exp\l(\frac{n^2}{2}( \frac{1}{n^2} \sum_{i\ne j} \log |z_i - z_j| - 
\frac{2}{n}\sum_i \log |1-z_i|  
+\frac{4}{n^2} \sum_i \log|1-z_i|) \r) \\
&& \times \exp\l(-3\sum_{i} \log |1-z_i|\r) 
{\bf 1}_{B_{n,k}}(z_1,\ldots,z_n).
\end{eqnarray*}
Note that 
\[\frac{1}{n^2}\sum_{i\ne j} \log |z_i - z_j|\le \frac{1}{n^2}\sum_{i\ne j} (\log_+ |z_i - z_j|)  =  \iint \log_+ |z-w| dL_n(z) dL_n(w) .\]
Thus, on  
the event 
        $\{L_n \in K_B^\complement\} 
        \cap \{|X_n|\le B\}\cap A_{n,k}$,
        we have that
$$       \frac{1}{n^2} \sum_{i\ne j} \log |z_i - z_j| - 
\frac{2}{n}\sum_i \log |1-z_i|+
\frac{4}{n^2} \sum_i \log|1-z_i|\leq -5B+4B=-B$$
and therefore on this event,
         $$f_{k,n}(z_1,\cdots,z_N)
         \leq \frac{1}{{\mathcal Z}_{n,k}} e^{-n^2B/2}
         \exp\l(-3\sum_{i} \log |1-z_i|\r){\bf 1}_{B_{n,k}}(z_1,\ldots,z_n).$$ 
 Thus, using \eqref{eq-Ynnew} and the constant $c$ in the statement of 
 the lemma,
 \begin{eqnarray*}
        && \P_n \l(\{L_n \in K_B^\complement\} 
        \cap \{|X_n|\le B\}\cap A_{n,k}\r)\\
        & \leq &
        \frac{1}{{\mathcal Z}_{n,k}} e^{-n^2B/2}
         \int \cdots \int \left(
         \left[\prod_{i=0}^n\frac{1}{|1-z_i|^{3}}\right]
         \wedge e^{3cn}
         \right)
                 d\el(z_1)\cdots d\el(z_k)
d\ell(z_{2k+1})\cdots d\ell(z_n).
\end{eqnarray*}
Lemma \ref{exp-tight} follows from substituting 
the last display in \eqref{eq-111413b} and performing the
integration.
\end{proof}
\begin{proof}[Proof of Lemma \ref{lem-xn}]
        By the argument in the proof of
        Lemma \ref{obr2}, we have that for any $j$ non-negative integer, 
        $$\frac1n \sum_{z_i: |1-z_i|\in [2^{-{j+1}},2^{-j}]}
        |\log|1-z_i||\leq j\cdot 2^{-j}\,.$$
        Thus, 
        \begin{equation}
                \label{eq-Ynnew1}
                Y_n=    \frac1n \sum_{z_i:|1-z_i|\leq 1} |\log |1-z_i||
        \leq \sum_{j=0}^\infty j2^{-j}\,.
\end{equation}
        In particular, for all $n$ large, 
        \begin{equation}
                \label{eq-111413a}
                \P_n(X_n\leq -1)=0\,.
        \end{equation}
        Next, we control the upper tail of $X_n$.
        We have
        \begin{eqnarray}
                \label{eq-11}
                P(X_n>B)&=&P(\sum_{i=0}^{n-1}\xi_i>(e^{Bn^2}-1)\xi_n)
                        \leq P(\sum_{i=1}^{n-1}
                        \xi_i>\frac12 e^{Bn^2}\xi_n)\nonumber\\
                        & =&
                        \int_0^\infty e^{-x}
                        P(
                        \sum_{i=1}^{n-1}
                        \xi_i>\frac12 e^{Bn^2}x)dx\,.
                \end{eqnarray}
                Using Chebycheff's inequality, we have
                        $$P(\sum_{i=1}^{n-1}
                        \xi_i>\frac12 e^{Bn^2}x)\leq 
                        e^{-\lambda e^{Bn^2}x/2}
                        \l[E(e^{\lambda \xi_1})\r]^{n-1}
                        \leq
                        \frac{e^{-\lambda e^{Bn^2}x/2}}{(1-\lambda)^n}\,.
                        $$
                        Choosing $\lambda=1/n$ and substituting 
                        in \eqref{eq-11} gives
                $$P(X_n>B)\leq
                 e\cdot \int_0^\infty e^{-x(1+e^{Bn^2}/2n)} dx 
                \leq 4e\cdot ne^{-Bn^2}\,.$$
                Combining the last display with
                \eqref{eq-111413a} completes the proof.
\end{proof}

\subsection{The Upper Bound}
\label{ubound}
Recall the notation 
$f(z,w)=\log |1-z| + \log |1-w| - \log|z-w| $.
We prove in this subsection the following.
\begin{lemma}
        \label{lem-UB}
        For any $\mu\in\p$,
 \begin{equation}
  \label{upb}
   \lim_{\eps \to 0} \varlimsup_{n \to \infty} \frac{1}{n^2} \log\, \P_n \l(  d(L_n,\mu) \le \eps \r) \le - \frac{1}{2} \iint f(z,w)d\mu(z)d\mu(w)\,.
 \end{equation}
 \end{lemma}
 Here, $d(\cdot,\cdot)$ is an arbitrary metric on $\M_1(\C)$ which is 
 compatible with the weak topology, e.g. the L\'{e}vy metric.

 \begin{proof}
Define the set of measures   \[E_n:= \{\nu \in \p: 
\frac{1}{n} \int \log |1-z| d\nu(z) \ge \frac12\iint  f(z,w)
d\nu(z)d\nu(w)\}.\] The set $E_n$ corresponds to a subset of $\cup_{k=0}^{\lfloor n/2 \rfloor}A_{n,k}$ which gives rise to empirical measures $\nu$ as described in Section \ref{intro}. By abuse of notation, we denote this set by $E_n$ as well.

An application of \eqref{eq-Xnnew} of Lemma \ref{lem-xn} gives 
the following.
\begin{proposition}
With notation as above,
$\P_n(E_n) \le 20 n
\exp \l( -\frac{1}{2}n^2 \int \int f(z,w) d\mu(z)d\mu(w)\r)$.
\end{proposition}

Now,
\[
\P_n \l(  d(L_n,\mu) \le \eps \r) \le \P_n \l( E_n^\complement
\cap \{ d(L_n,\mu) \le \eps \} \r) + \P_n(E_n).
\]

Therefore, 
\begin{align*}&\varlimsup_{n \to \infty} \frac{1}{n^2} \log \P_n \l(  d(L_n,\mu) \le \eps \r) \\
= & \text{Max} \l\{\varlimsup_{n \to \infty} \frac{1}{n^2} \log
\P_n (E_n),
\varlimsup_{n \to \infty} \frac{1}{n^2} \log \P_n \l( E_n^\complement
\cap  \{ d(L_n,\mu) 
\le \eps \} \r)\r\}.\end{align*} 
Since $\varlimsup_{n \to \infty} \frac{1}{n^2}\log \P_n (E_n)$ is 
bounded above by the desired upper bound, it remains to deal with 
$\varlimsup_{n \to \infty} \frac{1}{n^2} \P_n \l( E_n^\complement
\cap \{ d(L_n,\mu) \le \eps \} \r)$. 


We begin with 
\[ \P_n( E_n^\complement
        \cap \{ d(L_n,\mu)<\eps \} )  = \sum_{k=0}^{\lfloor n/2\rfloor}
        \frac{1}{ {\mathcal Z}_{n,k}}I_{k,n}^{\eps}, \] 
where \[ I_{k,n}^{\eps} = \int_{\{E_n^\complement
        \cap A_{n,k} \cap B_{n,k}
\cap \{d(W_n,\mu) \le \eps \} \}} \exp\l( \sum_{1\le i < j \le n} 
\mathrm{log} |z_i - z_j| - (n+1) \sum_{j=1}^n \mathrm{log}|1-z_j| \r) \]\[ 
d\el(z_1)\cdots d\el(z_k)d\ell(z_{2k+1})\cdots d\ell(z_n)\] where 
$W_n(z_1,\cdots,z_n)$ is the empirical measure $\frac{1}{n} \sum_{i=1}^n \del_{z_i}$.

We will upper bound $\varlimsup_{n \to \infty} \frac{1}{n^2} \log
I_{k,n}^{\eps}$ for each $0\le k \le \lfloor n/2\rfloor$, 
uniformly in $k$; by summing over $k$, 
this (together with \eqref{eq-znk})
will be sufficient for the overall upper bound on 
$\P_n\l(E_n^\complement \cap \{d(L_n,\mu) \le \eps \} \r)$.

For reasons similar to those encountered in the proof of exponential tightness,
we  write the integrand in  $I_{k,N}^{\eps}$ as 
\begin{equation} \label{e1} \exp\l( \sum_{1\le i < j \le n} 
        \mathrm{log} |z_i - z_j| - (n-1) \sum_{j=1}^n \mathrm{log}|1-z_j| + 
        \eps \sum_{j=1}^n \mathrm{log}|1-z_j|  \r) 
        \prod_{j=0}^n\frac{1}{|1-z_j|^{2+\eps}}. \end{equation}

Note that to upper bound the exponent in (\ref{e1}), it suffices to 
truncate $\log|z_i-z_j|$ from below and $\log|1-z_i|$ from above. 
To this end, we fix a big positive number $M$ and define the 
truncated function \[f_M(z,w)= f(z,w) \wedge M   .\]

The exponent in (\ref{e1}) is 
\begin{equation}
 \label{u1}
\mathcal{E}_n(z_1,\cdots,z_N) \le \frac{n^2}{2} \l( - \iint_{z \ne w} 
f_M(z,w) dL_n(z)dL_n(w) - \frac{2\eps}{n} \int \log|1-z|  dL_n(z)  \r)
\end{equation}
and  $\exp\l( \mathcal{E}_n(z_1,\cdots,z_n) \r)$ is 
integrated, for each fixed $k$, with respect to the measure 
\begin{equation} \label{bgm}\prod_{j=0}^n\frac{1}{|1-z_j|^{2+\eps}} 
        {\bf 1}_{B_{n,k}}(z_1,\ldots,z_n)
        d\el(z_1)\cdots d\el(z_k)d\ell(z_{2k+1})\cdots d\ell(z_n).\end{equation}
But \[ \iint_{z \ne w} f_M(z,w) dL_n(z)dL_n(w)= \iint 
f_M(z,w) dL_n(z)dL_n(w) - M/n\]
In the above equality, the $M/n$ term comes from the diagonal terms 
in the discrete sum $\iint f_M(z,w) dL_n(z) dL_n(w)$.

We handle (\ref{u1}) with the following proposition, whose proof is deferred
to the end of this section.
%
\begin{proposition}
\label{expon}
There exist $\del_M(\eps)>0$ and $c(M)>0$ 
such that for all $\nu \in \p$ such that $d(\nu,\mu)<\eps$ we have 
\[ \l|\iint f_M(z,w) d\nu(z)d\nu(w) - \iint f_M(z,w) d\mu(z)d\mu(w)\r| 
<\del_M(\eps) + c(M).\] where $\del_M(\eps) \to 0$ as $\eps \to 0$ for each fixed $M$ (bigger than some universal constant) and $c(M) \to 0$ uniformly in $\nu \in \p$ and $\eps$. 
\end{proposition}


Continuing with the proof of the upper bound,
we use Proposition
\ref{expon} in
(\ref{u1}) together with 
\eqref{eq-Ynnew} to write
\begin{eqnarray*}
        I_{k,n}^\epsilon&\leq& 
        C n \exp\l\{{4n^2(\del_M(\eps) +c(M)+ n^{-1}M)}\r\}  \\
        &&\times \exp\l\{ -\frac{n^2}{2} \l( \iint f(z,w) d\mu(z)d\mu(w) - 
\eps  \iint f(z,w) d\mu(z)d\mu(w) \r) \r\}\\
&&
\int \cdots \int \left(\prod_{i=0}^n\frac{1}{|1-z_i|^{2+\eps}}\wedge c^{(2+\epsilon)n}
         \right)
                 d\el(z_1)\cdots d\el(z_k)
d\ell(z_{2k+1})\cdots d\ell(z_n).
\end{eqnarray*}
The last integral is dominated by $e^{C(\epsilon)n}$ for appropriate 
$C(\epsilon)$.
Taking logarithm, dividing by $n^2$ and letting $n \to \infty$, $\eps \to 0$ and $M \to \infty$ (in that order) we get the desired upper bound (\ref{upb}).
\end{proof}

\begin{proof}[Proof of Proposition \ref{expon}]
The statement would follow immediately from the
definitions if $f_M$ was a bounded continuous functions. 
Although $f_M$ is not a bounded continuous function,  
$f_M$ is  clearly bounded above. We introduce
$g_M=f_M \vee (-M)$.   Note that switching with
$z'=1-z,w'=1-w$, we have $g_M(1-z',1-w')= (-M) \vee M 
\wedge (-\log|\frac{1}{z'}-\frac{1}{w'}|)$. 

Let $A_M$ be the set \[A_M:=\{(z,w): f_M(z,w) <-M \}.\] Clearly, \[\iint 
f_M(z,w) d\nu(z)d\nu(w) = \iint g_M(z,w) d\nu(z)d\nu(w) + 
\iint_{A_M} (f_M(z,w)+M)d\nu(z)d\nu(w) .\] 

We consider integration over the domain
 $|w'|\le |z'|$; by symmetry, the complementary domain can be handled similarly. Then on the set $A_M$, 
we have
$f_M(z,w) = -\log|1-\frac{w'}{z'}| + \log|w'|$. But since 
$|w'|\le |z'|$, we have $-\log|1-\frac{w'}{z'}|  \ge -\log 2$ and
therefore on the set $A_M$ we have \[\log |w'| \le -M +  \log 2\] and 
\begin{equation}
\label{control}
-\log 2 + \log|w'| \le f_M(z,w) \le -M.             \end{equation}
 Let $B_{M,\nu}$ be the event that for two i.i.d.  
 variables $(X,Y)$ sampled from $\nu$, the minimum satisfies 
 \[\log\l(\text{min}(|1-X|,|1-Y|)\r)\le -M +2.\] 
 Clearly, $A_M \subset B_{M, \nu}$. 
 From Lemma \ref{obr2}, we deduce that 
 \begin{equation} \label{bm1} \nu(B_{M,\nu})<c_1(M)\end{equation} 
 where $c_1(M)\to 0$ as $M\to \infty$ uniformly in $\nu$. 
 Furthermore, the same lemma implies that  
 \begin{equation}\label{bm2}\l| \int_{B_{M,\nu}} 
         \log \l(\text{min}(|1-z|,|1-w|)\r) d\nu(z) d\nu(w) \r|  
         <c_2(M)\end{equation} where $c_2(M)\to 0$ as $M\to \infty$
 uniformly in $\nu$. 

Combining (\ref{control}), (\ref{bm1}) and (\ref{bm2}) we get
\[ \iint_{A_M} (f_M(z,w)+M)d\nu(z)d\nu(w) =c_3(M,\nu) \] where 
$c_3(M,\nu)  \to 0$ as $M \to \infty$ uniformly in $\nu \in \p$. 
In other words, we have
\[ \l| \iint f_M(z,w) d\nu(z)d\nu(w) - \iint g_M(z,w) d\nu(z)d\nu(w) \r| \to 0 \] as $M\to\infty$, uniformly in $\nu \in \p$.

It remains to show that $\iint g_M d\nu\otimes d\nu \to \iint g_M d\mu 
\otimes d\mu$ as $\nu \to \mu$ for a fixed $M$. But this is true by definition since $g_M$ is a bounded continuous function.
\end{proof}

 \section{Convergence of potentials}
 \label{potential}
We equip  $\M_1(\C)$ with a metric $d$ compatible with the weak topology.
Recall that for $\mu\in \M_1(\C)$, $\LL_\mu$ denotes the logarithmic 
potential of $\mu$. The following theorem is standard - 
see e.g. \cite[Theorem 3.2.13]{Hor} for part (b). 
 \begin{theorem}
 \label{potconv}
Let $K\subset\subset \C$  be  compact set with non-empty interior.
Let $\h(K^{\c})$ denote the space of harmonic functions on 
$K^{\c}$, equipped with the topology of uniform convergence on compact sets. 
Further, let $f$ be a continuous function on $K$.  
  \begin{itemize}
  \item[(a)]  The map from $\mm(K) \to \h(K^{\c})$, given by  $\nu \mapsto \LL_\nu$, is continuous, and the same holds for the derivatives of $\LL_\nu$.  
  \item[(b)]  The map from $\mm(K) \to \R$, given by $\nu \mapsto \sup_{z\in K} \l(\LL_\nu(z) - f(z)\r) $, is upper semicontinuous. 
  \end{itemize}
 \end{theorem}

 \section{Approximation of measures}
 \label{meastop} 
Fix $\del>0$ and  let $K\subset\subset \C$ be defined by
\[K^{\c} = \{ z: |\mathrm{arg}(z)|<\del \} \cup 
\{ z: |z| < \del \} \cup \{ z: |z| > \frac{1}{\del} \}.\] 
We let 
  $\mm^{\mathrm{Sym}}(K)$ denote 
  the set of probability measures 
  supported on $K$ that are symmetric about the real line.  
%

  \begin{definition}
 Let $\wt{\mm}(K)\subset 
 \mm^{\mathrm{Sym}}(K)$
 denote the collection of symmetric probability measures
 $\mu$ such that 
	$\LL_\mu(z)<\LL_\mu(|z|)$ for $z\not\in \R_+$ and
 there exist real numbers
 $a=a(\mu),b=b(\mu)$ so that
 \begin{equation}
 \label{infinity}
  \LL_{\mu}(x+ i y) = \log |z| + b/x + O(|z|^{-2}), z =  x + i y \to \infty
 \end{equation}
and 
 \begin{equation}
  \label{zero}
  \LL_{\mu}(x+ i y)= \LL_{\mu}(0) + a x + O(|z|^2), z = x+ i y \to 0.
 \end{equation}
\end{definition} 
\noindent Note that 
 $\wt{\mm}(K)\subset \p$. 
 
 We are now ready to state Theorem \ref{convergence1},
whose proof builds on arguments in \cite{BES}. It
will be instrumental in obtaining the lower bound in Theorem
\ref{theo-Lihom}.
\begin{theorem}
 \label{convergence1}
 $\wt{\mm}(K)$ is a relatively open subset of $\mm^{\mathrm{Sym}}(K)$. 
 \end{theorem}
 
\begin{proof}
Fix $\mu \in \wt{\mm}(K)$ and set
$$\mm^{\mathrm{Sym},\mu,\eps}(K)=\{   
\nu \in \mm^{\mathrm{Sym}}(K):
d(\nu,\mu)<\epsilon\}\,.$$
  We will show that there is 
$\epsilon=\epsilon(\mu)$ 
so that
$\mm^{\mathrm{Sym},\mu,\eps}(K)\subset
\wt{\mm}(K)$.
Since 
$\LL_{\nu}(z)=\LL_{\nu}(\ol{z})$ on
$\mm^{\mathrm{Sym}}(K)$, 
we focus on proving that for $\eps$ small enough and all 
for $\nu\in\mm^{\mathrm{Sym},\mu,\eps}(K)$,
\begin{equation} \label{target}\LL_{\nu}(z) < \LL_{\nu}(|z|)\,,
\quad z\ne 0\,.\end{equation} The other properties 
in the definition of $\wt{\mm}(K)$ 
will follow from the proof of (\ref{target}).
  
We handle different 
$z$ according to whether their modulus $|z|$ is small,
large or intermediate, corresponding to the different ranges 
in the definition of $K$. Recall that
$\LL_{\nu}(z)$ is harmonic for $|z|<\del$, 
hence for $z=re^{i \t}$ and $0<r<\del$ we have the expansion
 \begin{equation}
  \label{expand}
  \LL_{\nu}(z)=\sum_{n=0}^{\infty} a_n^{\nu} r^n \cos n\t
 = a_0 + a_1^{\nu} r \cos \t + \Phi_{\nu}(z)
 \end{equation}
 where
%
\begin{equation}
 \label{tail}
 \Phi_{\nu}(z)=\sum_{n=2}^{\infty}a_n^{\nu}r^n \cos n\t.
\end{equation} 

By part (a) of Theorem \ref{potconv} applied to $\LL_\mu$ (and $\LL_\nu$)
in the region
$|z|<\delta$ where both are harmonic,
there exists $\epsilon_0=\epsilon_0(\mu)$ such that for $\eps<\eps_0(\mu)$,
$|\LL_{\nu}(z)|<C(\mu,r_0)$ for all $|z|\le r_0<\delta$ and
$\nu\in\mm^{\mathrm{Sym},\mu,\eps}(K)$.
Cauchy's inequality implies that $ |a_n^{\nu} r_0^n| \le 
C_1(\mu,r_0),$
leading to 
\[|\Phi_{\nu}(z)|\le C_2(\mu,r_0)r^2\,, \quad r < r_0/2.\] 
Recall (\ref{zero}). Invoking part (a) of Theorem 
\ref{potconv} for the first order partial derivative 
$\frac{\partial \LL}{\partial r}$, we have that $a_1^{\nu}> a/2$ as soon as 
$\eps$ is small enough. Thus,
for $r_1=r_1(\mu,\epsilon_0)$ small enough, one has
(\ref{target}) for all $|z|\le r_1$. 
This also verifies (\ref{zero}) for the measure $\nu$.

The case of large $|z|$ is treated similarly, 
using (\ref{infinity}) and the transformation 
\[ \LL_{\nu} (z) \to \log |z| + \LL_{\nu}(1/z). \] 
The upshot is that there exists $r_2=r_2(\mu)>0$
such that (\ref{target}) holds for $\nu
\in \mm^{\mathrm{Sym},\mu,\eps}(K)$, whenever $|z|>r_2$ and $\eps$ is
small enough. 
This also verifies (\ref{infinity}) for the measure $\nu$.
 
We are left to deal with the case $r_1 \le |z| \le r_2$. 
Define the sector
\[ \L_\beta:=\{z: |\mathrm{arg}(z)| \le \beta, r_1 \le |z| \le r_2  \} \] 
Since $\LL_\mu(z)<\LL_\mu(|z|)$ for all $z\ne 0$, 
one has
$\frac{\partial^2}{\partial \t^2} \LL_{\mu}(r)<0$ for $r>0$. 
Hence, there is a constant $c=c(\mu,r_1,r_2)>0$ so that
$\frac{\partial^2}{\partial \t^2} \LL_{\mu} < -c$ 
in some sector $\L_{\beta_0}$ with $\beta_0>0$,
that is disjoint from the set $K$.

Observe that 
$\frac{\partial \LL_\nu}{\partial \t}(r)=0$
because 
$\LL_\nu(re^{i\t})=\LL_\nu(re^{-i \t})$. 
Hence, on $\L_{\beta_0}$, we can write 
\[\LL_\nu(re^{i\t})=\LL_\nu(r) + 
\frac{1}{2}\l(\frac{\partial^2 \LL_\nu}{\partial \t^2}(r)\r)\t^2 + 
g_\nu(r,\t) \t^3,\] where $|g_\nu(r,\t)|\le \sup_{\L_{\beta_0}} \l| 
\frac{ \partial^3 \LL_\nu}{\partial \t^3 }(re^{i \t}) \r|.$  
By part (a) of Theorem \ref{potconv}, for $\eps$ small enough, 
we have on $\L_{\beta_0}$ the inequalities 
\[\frac{ \partial^2 \LL_\nu}{\partial \t^2 }(re^{i \t}) \le -c/2\] and 
\[ \sup_{\L_{\beta_0}} \l| 
\frac{ \partial^3 \LL_\nu}{\partial \t^3 }(re^{i \t}) \r| 
\le \sup_{\L_{\beta_0}} \l| \frac{ \partial^3 \LL_\mu}{\partial \t^3 }(re^{i \t}) \r| + 1. \] 
Therefore, we can find $0 < \beta \le \beta_0$  such that
on $\L_\beta$
we have  \[\frac{1}{2}\l( \frac{\partial^2 \LL_\nu}{\partial \t^2}(r)\r)\t^2 +
g_\nu(r,\t) \t^3 < 0 \] as soon $\eps<\epsilon_1=\epsilon_1(\mu)$, 
implying
that
(\ref{target}) holds in $\L_\beta$ 
for $\nu\in \mm^{\mathrm{Sym},\mu,\eps}(K)$.

Define the compact set \[ \wt{K}:=\{z: |\mathrm{arg}
(z)| \ge \beta, r_1 \le |z| \le r_2  \} \]
Note that $v(z)=\LL_\mu(|z|)$ is continuous in a neighborhood of $\R_+$,
and therefore 
the function $\LL_\mu(z) - \LL_\mu(|z|)$ is upper-semicontinuous on 
$\wt{K}$, and hence attains its maximum there. Since
$\LL_\mu(z) - \LL_\mu(|z|)<0$ on $\wt{K}$, it follows that
there exists a  $c'>0$ such that  $\LL_\mu(z) - \LL_\mu(|z|)< - c'$ on 
$\wt{K}$. 

Next, observe that the function
$(z,\nu)\mapsto \LL_\nu(|z|)$ is continuous on
the compact set
$(r_1 \le |z| \le r_2)\times  {\mm}^{\rm{Sym}}(K)$. Therefore,
for $0<\epsilon<\epsilon_2(\mu)$,
 one has $|\LL_\mu(|z|) - \LL_\nu(|z|)|<c'/2$ for $\nu\in \mm^{\mathrm{Sym},\mu,\eps}(K)$.
 As a result, for such $\nu$ we get
 $|\LL_\mu(|z|) - \LL_\nu(|z|)|<c'/2$ and therefore
\[ \sup_{\wt{K}} \l( \LL_\nu(z) - \LL_\nu(|z|)  \r) \le \sup_{\wt{K}} \l( \LL_\nu(z) -\LL_\mu(|z|)  \r)  + c'/2 . \]
We bound the right hand side from above by invoking part (b) of 
Theorem \ref{potconv} on the set $\wt{K}$ with $f(z)=\LL_\mu(|z|)$. 
This yields the existence of $\epsilon_3(\mu)>0$ so that,
 for $\nu\in \mm^{\mathrm{Sym},\mu,\eps}(K)$ and $\eps\leq\eps_3(\mu)$,
\[\sup_{\wt{K}} \l( \LL_\nu(z) - \LL_\nu(|z|)  \r) \le \sup_{\wt{K}} \l( \LL_\mu(z) -\LL_\mu(|z|)  \r) + c'/2 \le -c'/2. \]
This establishes (\ref{target}) for such $\nu$, 
thereby completing the proof of the theorem.  
\end{proof}
We now introduce the notion
of an $\eps$-perturbation of an atomic measure $\nu$.
\begin{definition}
	Fix $\eps>0$ and 
	an atomic probability measure $\nu\in \mm^{\mathrm{Sym}}$ 
	with $k$ distinct atoms 
$\{\nu_1,\cdots,\nu_k\}$ of equal mass. 
Then 
$U(\nu,\eps)$ denotes 
the set of atomic probability measures $\mu\in \mm^{\mathrm{Sym}}$ with  $k$
distinct atoms $\{\mu_1,\cdots,\mu_k\}$ of
equal mass,
such that $\max_{1\le i \le k}\|\nu_i - \mu_i\|_{\i}<\eps$. 
Further, if $\nu_i$ is real the $\mu_i$ is also required to be real.
\end{definition}

We end this section with two remarks, both of which will 
come in handy in the proof of Proposition \ref{stage1-1}.
\begin{remark}
\label{impact}
The importance of Theorem \ref{convergence1} is the following
approximation result.
Let $\mu \in \p$ be a probability measure supported on $K$. Let 
$\mu_k$ be a sequence of atomic measures with distinct atoms and equal 
mass on each atom, such that $\mu_k \to \mu$ in $\M_1(\C)$. 
Let $\{\a_k\}_k$ be any sequence of positive numbers such that 
$\a_k \to 0$ as $k \to \i$. Then, for sufficiently large $k$, 
we have $U(\mu_k, \a_k) \subset \p$.
\end{remark}
\begin{remark}
 \label{coeff}
As in the proof of Theorem \ref{convergence1},
the convergence of $\frac{\partial^2 \LL}{\partial \t^2}$ 
in 
$\h(K^{\c})$ implies that there exists $\epsilon_5(\mu)$ such that for
$\eps<\epsilon_5(\mu)$,
the coefficient $a_{1}^\nu > 0$  if $\nu\in\mm^{\mathrm{Sym},\mu,\eps}(K)$.
\end{remark}

\section{The Lower Bound}
\label{lbound}
Our goal in this section is to prove the following.
Recall that $f(z,w)=\log |1-z| + \log |1-w| - \log|z-w| $.
\begin{lemma}
        \label{lem-LB}
For any $\mu \in \p$ with $I(\mu)<\infty$,
 \begin{equation}
  \label{lowb}
   \lim_{\eps \to 0} \varliminf_{n \to \infty} \frac{1}{n^2} \log 
   \P_n \l(  d(L_n,\mu) \le \eps \r) \ge  - I(\mu) = 
   - \frac{1}{2}  \iint f(z,w)d\mu(z)d\mu(w).
 \end{equation}
 \end{lemma}

The proof of Lemma \ref{lem-LB} proceeds by several approximation steps.
Those are detailed in the rest of the subsection, with several technical
propositions deferred to Sections \ref{sec-proofstage1-1} and \ref{erapprox}.
The approximation is inspired by \cite{BES}.

\subsection{A dense subclass $\d \subset \p$}
We introduce a
dense (in the metric of $\M_1(\C)$)
subset $\d \subset \p$  such that for any  measure $\mu \in \p$ there is
a sequence $\{\mu_m\}_{m=1}^{\infty}$ from $\d$ such that 
\begin{itemize}
        \item[(i)] $\mu_m \to \mu  $ as $m\to\infty$ (convergence in the weak topology of $\M_1(\C)$),
        \item[(ii)] $I(\mu_m) \to I(\mu)$ as $m \to \infty$, 
        \item[(iii)] For any $\nu \in \d$, the estimate  (\ref{lowb}) holds.
\end{itemize}

$\D \subset \M_1(\C)$  will consist of those
probability measures $\nu$ satisfying:
\begin{itemize}
        \item[(a)] 
                $\supp(\nu)$ is contained in
                 a compact annulus centered at the origin.
         \item[(b)] $\supp(\nu)$ is disjoint from
                 a cone whose apex is at the origin and 
		 whose axis is the positive ray
                 $\R_+$. 
         \item[(c)] $\widehat{\LL}_{\nu}(z) < 
                 \widehat{\LL}_{\nu}(|z|) $ 
                 for each $z \in \C\setminus \bar \R_+$, 
                 while $\widehat{\LL}_{\nu}(z)=
                 \widehat{\LL}_{\nu}(0)+
                 a \Re{z} + O(|z|^2)$ as $|z|\to 0 $ 
                 and $\widehat{\LL}_{\nu}(z)=\log |z| + b{\Re(1/z)} + 
                 O(1/{|z|^2})$ as 
                 $|z|\to \infty$, with both $a$ and $b$ being positive.
         \item[(d)] $\nu$ has a bounded density with respect to the Lebesgue measure. 
\end{itemize}

We will use the \cite{BES} 
construction (with a slight modification of their step 5 and an additional step 6),
in order
to construct, for any $\mu\in\p$ with $I(\mu)<\infty$,
a sequence $\mu_\eps\in \d$ with $\mu_\eps\to\mu$, and further ensuring that 
 $I(\mu_\eps)\to_{\eps\to 0} I(\mu)$.
For the sake of completeness, we give a detailed account of the 
construction, its modifications and approximation properties
in Section \ref{erapprox}. 
That is,
we prove in Section 
\ref{erapprox} the following proposition.
\begin{proposition}
        \label{prop-era}
        For any $\mu \in \p$ with $I(\mu)<\infty$, there
        exists a sequence $\mu_\eps\in \d$ so that $\mu_\eps\to_{\eps\to
        0} \mu$
        and $I(\mu_\eps)\to_{\eps\to 0} I(\mu)$.
\end{proposition}

Equipped with 
Proposition \ref{prop-era} and in view of properties (i)--(iii) above, 
Lemma
\ref{lem-LB} is an immediate  consequence of the following proposition 
and the local nature of the large deviations lower bound.
\begin{proposition}
        \label{prop-LB}
        The lower bound \eqref{lowb} holds for any $\mu\in \d$.
\end{proposition}
The rest of this section is devoted  to the proof of Proposition
\ref{prop-LB}.

\subsection{Proof of Proposition \ref{prop-LB}.}
\label{lowd}
The proof proceeds in several approximation steps.
We fix throughout a $\mu\in \d$.
We first construct in Proposition \ref{stage1-1} 
a sequence of polynomials with positive coefficients
whose empirical measure of zeros $\mu_k$
approximates $\mu$, at the same time ensuring 
their (discrete) logarithmic energies approximate 
the logarithmic energy of $\mu$. We then show in Proposition
\ref{stage1-2} that it is
enough to prove the lower bound for balls centered at the 
 $\mu_k$. The proof of
the later lower bound is then obtained by first
constructing appropriate neighborhoods of $\mu_k$,
and then lower bounding the probability by lower bounding
the density of $\P_n$ on these neighborhoods.

\noindent
\textbf{Stage I: Reduction to  atomic measures }
\newline
We introduce the discrete version of logarithmic 
energy for atomic measures with distinct atoms, as
follows.
\begin{definition}
\label{sig}
For an atomic measure $\mu$ with $k$ distinct atoms
$\{z_i\}_{i=1}^k$ of equal mass, let 
\[\S(\mu)=\frac{1}{k^2} \sum_{i \ne j} \log |z_i -z_j|\] 
denote the discrete logarithmic energy. 
\end{definition}
By an abuse of notation, we will also use the same notation
$\S(P)$ where $P$ is a polynomial (with distinct zeros); in 
that case,
the atomic measure being considered is the empirical measure 
of the zeros of $P$. 
We begin with the following proposition.
\begin{proposition}
\label{stage1-1}
For each $\mu\in \d$ one may find a sequence of monic polynomials $\{P_k\}$,
with empirical measure of zeros $\{\mu_k\}$,
satisfying the following properties.
\begin{itemize}
\label{poly}
\item[(i)] $P_k$ has positive coefficients.
\item[(ii)] $\mu_k\in \p$, $\mu_k \to \mu$ in $\M_1(\C)$ and 
        $\S(\mu_k) \to \Sigma(\mu)$ as $k \to \infty$.
\item[(iii)] $\mu_k$ has $a(k)$ distinct atoms, 
        none of which is real,
        and $\mu_k$  puts equal mass of $1/a(k)$ on each atom, 
        with $a(k) \to \infty$ as $k \to \infty$. In particular,
	$a(k)$ is even. Further, all atoms of $\mu_k$ are within distance
	$1/a(k)$ from the support of $\mu$.
\item[(iv)] $P_k$ is of degree $d(k)=a(k)h(k)$ for some positive 
        integer $h(k)$.
\item[(v)] The distance between any two distinct atoms of $\mu_k$ is at least $a(k)^{-2}$.
\item[(vi)] For $k$ large enough, $U(\mu_k,\frac{1}{2}a(k)^{-2}) \subset \p$. 
\end{itemize}
\end{proposition}
\noindent
The proof of Proposition 
\ref{stage1-1} is given in Section 
\ref{sec-proofstage1-1}.

The importance of the $\{P_k\}$ in Proposition 
\ref{stage1-1} lies in the following proposition.
\begin{proposition}
\label{stage1-2} 
To obtain (\ref{lowb}) for some  $\mu\in \d$, 
it suffices to prove that with
$\mu_k$ as in Proposition \ref{stage1-1}, for each $\eps>0$ and all large enough $k$ (depending on $\eps$), we have 
\begin{align} \label{ineqk1}  & \varliminf_{n \to \infty}
        \frac{1}{n^2}\log \P_n\l( d(L_n,\mu_k) \le \eps  \r)\nonumber\\ 
        \ge &- \frac{1}{2} \l( \int \log|1-z|d\mu_k(z) + 
        \int \log|1-z| d\mu_k(w) - \Sigma_{\mathrm{a}}(\mu_k) \r) + o_k(1).
\end{align}
Here $o_k(1)$ denotes a quantity which converges to 0 as $k \to \i$, all other variables being held fixed.
\end{proposition}
\begin{proof}[Proof of Proposition \ref{stage1-2}.]
Given $\eps >0$, we have for all $k$ large enough the inclusion of 
sets \[\{\nu: d(\nu,\mu_k) <\eps/2\} \subset \{\nu:d(\nu,\mu)<\eps\}.\]
This implies that given $\eps>0$, we have for all large enough 
$k$ \begin{align}  \label{ineqk22} &&\varliminf_{n \to \infty} 
        \frac{1}{n^2} \log \P_n \l(  d(L_n,\mu) \le \eps \r) 
        \ge \varliminf_{n \to \infty} \frac{1}{n^2}\log \P_n\l(
        d(L_n,\mu_k) \le \eps/2  \r) \nonumber \\ &&\ge - 
        \frac{1}{2} \l( \int \log|1-z|d\mu_k(z) + \int \log|1-w| d\mu_k(w) - 
        \Sigma_{\mathrm{a}}(\mu_k) \r) + o_k(1)  .\end{align}
Let $\mathrm{Supp}^\delta(\mu)$ denote the (closed) $\delta$-blowup of the 
support of $\mu$. For $\delta<\delta_0(\mu)$,
we have that $z\mapsto \log|1-z|$ is continuous on 
$\mathrm{Supp}^\delta(\mu)$, while by
property (iii) of  Proposition \ref{stage1-1},
for all $k$ large,
the support of each $\mu_k$ is contained in 
$\mathrm{Supp}^\delta(\mu)$. Thus,
$\mu_k \to \mu$ in $\M_1(\C)$ 
implies that $ \int \log|1-z|d\mu_k(z) \to  \int \log|1-z|d\mu(z)$. 
Moreover, by property (ii), 
we have that $\S(\mu_k) \to \Sigma(\mu)$ as $k\to\infty$.  
Letting now $k\to\infty$ first and then $\eps\to 0$ 
in (\ref{ineqk22}) yields (\ref{lowb}).
\end{proof}

\noindent
\textbf{Stage II: The Neighbourhoods $\n_n(\eps,k)$ of $\mu_k$}

In this stage we will fix $\eps>0$ and $k$ (to be chosen large enough depending on $\eps$) and define  
suitable $n$-dependent
neighbourhoods $\n_n(\eps,k)$
of $\mu_k$ in $\M_1(\C)$,  which are contained in the 
set $\{\nu: d(\nu,\mu_k)<\eps \}$. We begin by introducing 
an approximation radius $\rho(k)$.

\begin{proposition}
 \label{ro}
Let $\rho(k)=\frac{1}{4k}a(k)^{-2}$. 
For all $k$ large, we have:
 \begin{itemize}
  \item[] (i) $U(\mu_k,\rho(k)) \subset \p$.
  \item[] (ii) $d(\mu_k,\nu)<\eps/2$ for any  $\nu \in U(\mu_k,\rho(k))$.
  \item[] (iii) $\l| \int \log|1-z|d\mu_k - \int  \log|1-z|d\nu(z) \r| <1/k$ for each $\nu \in U(\mu_k,\rho(k))$.
  \item[] (iv) $\rho(k) < \frac{1}{k} \cdot \min_{i\neq j} |z_i -z_j|.$
\end{itemize}
\end{proposition}

\begin{proof}
 These properties follow easily from the definition of the measure $\mu_k$, that of the neighbourhoods $U$, and the fact that the minimal separation between two distinct atoms of $\mu_k$ is at least $a(k)^{-2}$.
\end{proof}

\begin{definition}
 Let $\nu$ be an atomic probability measure with $l$ atoms $\{\nu_1,\cdots, \nu_l\}$ and equal mass $1/l$ on each. Then by $P_{\nu}$ we denote the monic polynomial \[P_{\nu}(z)=\prod_{j=1}^l(z-z_j).\]
\end{definition}

Before stating the next proposition, we recall the notation that $pol_+$ denotes the set of all polynomials with positive coefficients.

\begin{proposition}
Let $\mu_k$  be as in Proposition \ref{poly} and $\rho(k)$ be as in Proposition \ref{ro}. 
\begin{itemize}
 \item[] \text{(i)} There exists a positive integer $m(k)$ such that for any $\nu \in U(\mu_k,\rho(k))$, we have $P_{\nu}^{m} \in pol_+$ for all $m \ge m(k)$.
  
 \item[] \text{(ii)} There exists $\t(k)>0$ such that for any $\nu \in U(\mu_k,\rho(k))$, the following is true: for every  $t \ge m(k)$ and distinct measures $\{\nu(i)\}_{i=1}^t \subset U(\nu,\t(k))$, we have $\prod_{i=1}^t P_{\nu(i)} \in pol_+$. We can find such $\t(k)$ so that $\rho(k)/\t(k)$ is an even integer.
\end{itemize}
\end{proposition}

\begin{proof}
 Since $U(\mu_k,\rho(k)) \subset \p$, for every 
 $\nu \in U(\mu_k,\rho(k))$ we have an integer $m=m(\nu)$ such that $P_{\nu}^t \in pol_+$ for all $t \ge m$, by De Angelis' Theorem 
 (\cite{DA}, quoted as Theorem \ref{dat} in the Appendix and as Theorem A in \cite{BES}).

Recall that all measures in $U(\mu_k,\rho(k))$ possess $a(k)$ distinct 
atoms, belong to $\mm^{\mathrm{Sym}}$, and possess the same number of
real atoms. Therefore, we can equip
$U(\mu_k,\rho(k))$ with the induced topology in $\R^{a(k)}$; in the rest of 
this proof, we work with this topology.

For any $t \ge m$,  one can find a small 
neighbourhood $U(\nu,\del_\nu(t))$ such that $P_\gamma^t \in pol_+$ 
for every $\gamma \in U(\nu,\del_\nu(t))$, 
because the coefficients of a monic polynomial depend continuously 
on its roots. If we define $\del_\nu=\min\{\del_\nu(m),\cdots,
\del_\nu(2m-1)\}$, then any $\gamma \in U(\nu,\del_\nu)$ 
satisfies $P_\gamma^t \in pol_+$ for every 
$t \ge m(\nu)$ (this can be seen by expressing $t$ as a 
multiple of $m(\nu)$ plus a remainder). Now, $\ol{U(\mu_k,\rho(k))}$ 
is a compact set which is covered by such $U(\nu,\del_\nu)$-s, 
hence it admits a finite subcover $\{U(\nu(i),\del_{\nu(i)})\}_{i=1}^M$. 
Now $m(k)=\max\{m(\nu(i));i=1,\cdots,M\}$ suffices for part (i).  
 
Proceeding on similar lines to part (i), for every $\nu \in
U(\mu_k,\rho(k))$, we can obtain a neighbourhood $U(\nu,\del_\nu)$ such 
that for every  $t \ge m(k)$ and distinct measures 
$\{\nu(i)\}_{i=1}^t \subset U(\nu,\del_\nu)$, we have 
$\prod_{i=1}^t P_{\nu(i)} \in pol_+$. Indeed,
since $P_{\nu}^{s} \in pol_+$,
for each $s \ge m(k)$, by the continuity of the 
coefficients in the zeros,
we can get a neighbourhood $\del_\nu(s)$ that works for that $s$; since
this is true for each $m(k)\le s \le 2m(k)-1$, take the minimum over $s$ 
in this range, and observe that one can write
any larger $t$ as a multiple of $m(k)$ plus a remainder . 
 
 Now, we can cover the compact set $\ol{U(\mu_k,\rho(k))}$
 with the neighbourhoods $U(\nu,\del_\nu/2)$, therefore there is a finite subcover $\{U(\nu(i),\del_{\nu(i)}/2)\}_{i=1}^M$. Then \[\t(k)=\min\{\del_{\nu(1)}/2,\cdots,\del_{\nu(M)}/2\}\] suffices for part (ii).
\end{proof}

\begin{definition}
 \label{subdiv}
 Let $\{z_1,\cdots,z_{a(k)}\}$  be the atoms of $\mu_k$. For $z_i, 1\le i \le a(k),$ define $B^i$ to be the closed $L^{\infty}$ ball of radius $\rho(k)$ around $z_i$. Divide each such $B^r$ into $s(k)^2=\l(\rho(k)/\t(k) \r)^2$ squares $\{B_{ij}^r\}_{i,j=1}^{s(k)}$ of sidelength $\t(k)$ each such that any two such squares overlap at most at the boundaries, and their sides are parallel to the axes. Let $z_{ij}^r$ be
 the center of $B_{ij}^r$, and note that $z_{ij}^r$ is not real.
\end{definition}


\begin{definition}
 \label{gms}
 Let $\gamma_{ij}$ denote the measure $\sum_{r=1}^{a(k)} \frac{1}{a(k)} \del_{z_{ij}^r}$.
 \end{definition}

 Let $m=\lfloor n / a(k)s(k)^2 \rfloor$.
Define the set of atomic measures 
\[S^{(m)}(P_k):= 
\bigg\{\nu: \frac{1}{ms(k)^2}\sum_{l=1}^{m}\sum_{i,j=1}^{s(k)}\nu_{ij}(l); 
\, \, 
\nu_{ij}(l) \in U(\gamma_{ij},\t(k)) \, \text{with distinct, 
non-real atoms } \bigg\} .\]

Fix a bounded interval $\I$ of length $<1$ on the negative real  line  
such that $\I$ is also bounded away from the support of 
$\mu_k$ by a distance $\ge 2$. Define the set of measures 
\[ \Upsilon^{(m)} :=\l\{\frac{1}{n-ma(k)s(k)^2} 
\sum_{i=1}^{n-ma(k)s(k)^2}\del_{\beta_i}: 
\beta_i \text{ are distinct numbers } \in \I   \r\}.\]

Finally, define $\n_n(\eps,k)$ to be the set of measures 
\[\n_n(\eps,k):=\l\{\frac{ma(k)s(k)^2}{n}\nu_1 + \l( 1- \frac{ma(k)s(k)^2}{n}\r)\nu_2: 
\nu_1 \in S^{(m)}(P_k), \nu_2 \in \Upsilon^{(m)} \r\}.\] 
Note that all measures in $\n_n(\eps,k)$ possess precisely $(n-ma(k)s(k)^2)$
real atoms.
Since $0\le n-ma(k)s(k)^2 \le a(k)s(k)^2$ (which is fixed since we are considering 
$k$ to be fixed), for large enough $n$ we have 
$d(\mu_k,\nu)<\eps$ for all measures $\nu \in \n_n(\eps,k)$.
Therefore,
\begin{equation} \label{ineqk2}
 \varliminf_{n \to \infty} \frac{1}{n^2}\log \P_n\l( d(L_n,\mu_k) \le \eps
 \r) \ge  \varliminf_{n \to \infty} \frac{1}{n^2}\log \P_n
 \l(L_n \in  \n_n(\eps,k) \r). 
\end{equation}

\begin{remark}
 \label{struct}
Each $\nu \in \n_n(\eps,k)$ has the following structure:
 \begin{itemize}
 \item $\nu$ has $n$ distinct atoms, each having equal mass $1/n$.
 \item The atoms of $\nu$ are the disjoint
         union of $(a(k)+1)$ subsets as follows:
 \begin{itemize}
  \item $\L^i(\nu):=\{w:w \text{ an atom of } \nu, \|z_i-w\|_{\i}\le \rho(k) \}, 1\le i \le a(k) $, with $|\L^i(\nu)|=ms(k)^2$ for each $i\in \{1,\cdots,a(k)\}$.
  \item $\L^0(\nu):=\{w: w \text{ an atom of } \nu, w \in \I \}$ with $|\L_0(\nu)|=n-ms(k)^2a(k)$. 
 \end{itemize}
 \item In each $\L^r(\nu), r \ge 1$, the atoms are the union of $s(k)^2$ disjoint subsets $\{\L_{ij}^r\}_{i,j=1}^{s(k)}$, with $z \in \L_{ij}^r$ if and only if $\|z-z_{ij}^r\|_{\i}\le \rho(k)/2s(k)$.
\end{itemize}
Conversely, every collection of $n$ points satisfying the above structure 
has the property that the corresponding empirical measure is 
in $\n_n(\eps,k)$.
\end{remark}

For each $\nu \in \n_n(\eps, k)$, we define the atomic measures 
\[\nu(r) := \frac{1}{|\L^r(\nu)|}\sum_{w \in \L^r(\nu)} \del_{w} \,,\quad
0\le r \le a(k) \]
and 
\[\nu(r;i,j):=  \frac{1}{|\L^r_{ij}(\nu)|}\sum_{w \in \L^r_{ij}(\nu)} \del_{w} 
\,,\quad 1\le r \le a(k), \,1\le i,j \le s(k)\,.\]

\noindent 
\textbf{Stage III: A good subset $\wn_n(\eps,k) \subset \n_n(\eps,k)$ 
and completion of the proof  of Proposition \ref{prop-LB}. }

\noindent
We introduce a subset
$\wn_n(\eps,k) \subset \n_n(\eps,k)$, and estimate
its volume in Proposition \ref{ngood}. We then use the estimate
to complete the proof of Proposition \ref{prop-LB}. We recall that for a set $B$, the quantity $\G(B)$ denotes the logarithmic energy of the uniform measure on $B$ (when it exists).

\begin{definition}
\label{ntilde}
For $1\le i,j \le s(k)$ and $1\le r \le a(k)$, recall that $B_{ij}^r$ is the $L^{\i}$-ball 
of radius $\t(k)$ centered at 
$z_{ij}^r$ and $B^r=\cup_{i,j=1}^{s(k)^2}B^r_{ij}$. 
Define $B^0$ to be the interval $\I$.

Define the set of atomic measures 
$\wn_n(\eps,k) \subset \n_n(\eps,k)$ as follows: 
\[\wn_n(\eps,k):=\bigg\{\nu \in \n_n(\eps,k) :  \S(\nu(r)) > 2\Sigma(B^r) 
 \text{ for each } 0 \le r \le a(k) \bigg\}.\]
By an abuse of notation, we also denote by $\wn_n(\eps,k)$ 
the subset of $\C^q \times \R^{n-2q}$ induced by
the atoms of  measures $\nu \in \wn_n(\eps,k)$ in the manner 
described in the introduction.
\end{definition}
Note that $B^r_{ij}\cap \R=\emptyset$ for all $r,i,j$.

With this definition, we have the following proposition, whose proof
is given at the end of this subsection.
In the statement of the proposition, 
by $\text{Volume}$ we mean the Euclidean volume.
\begin{proposition}
 \label{ngood}
 With notation as above, we have
\begin{itemize}
        \item[(i)] \text{Volume}($\wn_n(\eps,k)$)$\ge 
                \frac{1}{2^{a(k)+1}}\t(k)^{ma(k)s(k)^2}|\I|^{a(k)s(k)^2}$.
        \item[(ii)] For each $\nu \in \wn_n(\eps,k)$, 
                we have \[\S(\nu) \ge \S(\mu_k) + \log|1-\frac{2}{k}| +  
                \frac{2 \log \rho(k)}{a(k)} + \frac{C(k)}{n^2} \] 
        \item[(iii)] $\l| \int \log|1-z|d\mu_k - 
                \int  \log|1-z|d\nu(z) \r| <1/k$ for each 
                $\nu \in \wn_n(\eps,k)$. 
\end{itemize}
 \end{proposition}
 We can now complete the proof of Proposition \ref{prop-LB}.
\begin{proof}[Proof of 
                Proposition \ref{prop-LB}.] 
In what follows, we will use the notation $\wn_n$ to denote 
$\wn_n(\eps,k)$, unless explicitly mentioned otherwise.
We have
\begin{eqnarray*}
&&       P_n \l( \wn_n \r) 
         \geq \frac{1}{{\mathcal Z}_{n,q}}\times \\
         &&
\int_{\wn_n}  \exp \l\{\frac{n^2}{2} \l( \S(\w) - 2\frac{n+1}{n^2} 
\sum_{i=1}^n \log|1-w_i| \r) \r\} \prod_{i=1}^{q}
d\el(w_i) 
\prod_{j=0}^{n-2q-1}d\ell(w_{N-j}), 
\end{eqnarray*}
where
we recall that
each measure in $\wn_N$ has $n-2q$ real atoms, and $\w$ 
is the empirical measure corresponding to the set of atoms $\{w_i\}_{i=1}^n$.

By  Proposition \ref{ngood}, the exponent in the last integral is 
lower bounded (uniformly over all measures in $\wn_n$) by 
\begin{align*}&\mathcal{E}_n = \frac{n^2}{2} \times \\  &  \l(\S(\mu_k) + \log|1-2/k| +  
\frac{2 \log \rho(k)}{a(k)} + \frac{C(k)}{n^2} -(1+\frac{1}{n}) 
\cdot 2\int \log|1-z|d\mu_k(z) -\frac{2}{k}(1+\frac{1}{n}) \r) . \end{align*} 
Hence, 
\[ \P_n( \wn_n) \ge \exp \l(o(n^2)+
\mathcal{E}_n \r) \text{Volume}(\wn_n)  .\]
From Proposition \ref{ngood} it follows that $ \frac{1}{n^2}\log
\text{Volume}(\wt{\n}_n) \to 0$ as $n \to \infty$.
Therefore,
\begin{align*}
 & \varliminf_{n \to \infty}  \frac{1}{n^2}{\log  \P_n(\wn_n(\eps,k))} \\
& \ge \frac{1}{2}\l(\S(\mu_k) + \log|1-2/k| +  
\frac{2\log \rho(k)}{a(k)} - 2\int \log|1-z|d\mu_k(z) -\frac{2}{k} \r) .
\end{align*}
Since $\rho(k)=a(k)^{-2}/4k$, we obtain (\ref{ineqk1}).
Proposition \ref{stage1-2} now implies that the proof of (\ref{lowb}) is complete.
\end{proof}

We finally complete the proof of 
Proposition \ref{ngood}. We need the following notion.
\begin{definition}
The mutual energy of
$\nu^1,\nu^2 \in \M_1(\C)$ is defined as
$\Sigma(\nu^1,\nu^2)=\int \int \log |z-w| d\nu^1(z) d\nu^2(w)$, whenever the integral is well defined and finite.

For two atomic measures $\nu^1$ and $\nu^2$ with 
disjoint atoms (having equal mass) 
$\{\nu_1^1,\cdots,\nu_s^1\}$ and $\{\nu_1^2,\cdots,\nu_t^2\}$, define 
\[ \S(\nu^1,\nu^2)=\frac{1}{st}\sum_{i=1}^s\sum_{j=1}^t \log|\nu^1_i - \nu^2_j|. \]
\end{definition}

\begin{proof}[Proof of Proposition \ref{ngood}.]
We begin with part (i). Recall that for $i=1,\ldots,a(k)$,
$B^i$ does not intersect the real axis, and therefore the same is true of the sets $B^r_{ij}$.
Recall also that for all $r,i,j$, the cardinalities $|\L^r_{ij}|=m$ are all equal, so are $|\L^r(\nu)|=ms(k)^2=:p$ for all $r \ge 1$ . 


Let a p-tuple $\gamma=(\gamma_1,\cdots,\gamma_p)$ be sampled from $B^r$ as follows. Enumerate the $\{B_{ij}^r\}$-s in an arbitrary linear order as $D_1,\cdots,D_q$, where $q=s(k)^2$ (therefore, we have the relation $p=mq$; of course $p$ and $q$ both depend on $k$). Divide the co-ordinates of $\gamma$ into $q$ consecutive blocks $\{\b_i\}_{i=1}^q$ of size $m$ each. We sample $\gamma$ by sampling $m$ points uniformly at random from $D_i$ to constitute the block $\b_i$, the sampling being independent across all co-ordinates of $\gamma$. We will denote the sub-vector corresponding to the co-ordinates $\b_i$ by the symbol $\gamma^i$.   

By abuse of notation,
we still use
$\gamma$ to denote
the atomic measure $\frac{1}{p} \sum_{j=1}^p \del_{\gamma_j}$ and $\gamma^i$ to denote the atomic measure $\frac{1}{m}\sum_{j \in \b_i} \del_{\gamma_j}$. 


Observe that as probability measures, we have $\gamma=\frac{1}{q}\sum_{i=1}^q \gamma^i$. Let $U_i$ be the uniform measure on $D_i$, and let $U$ be the uniform measure on $B^r$. Then $U=\frac{1}{q}\sum_{i=1}^q U_i$.

Note that $\E[\S(\gamma^i)]=\l(1-1/m\r)\G(U_i)$ and $\E[\S(\gamma^i,\gamma^j)]=\G(U_i,U_j)$. Since all the energies under consideration are negative, we obtain from the 
identities in the last paragraph that $\E[\S(\gamma)]>\G(U)$, and $\S(\gamma)$ is a negative random variable a.s. So, $\E[|\S(\gamma)|] < |\G(U)|$.

Hence under a $\gamma$ randomly chosen from $B^r$ under the measure described above, we have $\P\l( |\S(\gamma)| >2 |\G(U)|  \r) < 1/2$ by Markov's inequality. But this implies that with probability $> 1/2$ (under our specific scheme of sampling $\gamma$), we have $|\S(\gamma)| < 2 |\G(U)|$, or equivalently, $\S(\gamma) > 2 \G(U)$ considering signs.

We can rewrite the last statement as:
\[ \mathrm{Volume}\{\gamma \in (B^r)^m: \S(\gamma) > 2\G(U) \} > \frac{1}{2}\prod_{i=1}^q \mathrm{Volume}(D_i)=
\frac{1}{2}\t(k)^{2q} \]

Note that $\G(U)=-\log \t(k)(1+o(1))$ as $k \to \infty$.

A similar argument with the atoms in 
$\L_0(\nu)$ for $\nu \in \n_n$ implies that 
\[\mathrm{Volume}\{\gamma \in \I^{n-ma(k)s(k)^2}: \S(\gamma)>2\G(\I)\}>
\frac{1}{2}\mathrm{Volume}(|\I|^{n - ma(k)s(k)^2}).\]
Definition \ref{ntilde} and Remark \ref{struct} now 
imply part (i) of Proposition \ref{ngood}.

We next turn to the proof of part (ii).
For $\nu \in \n_n(\eps,k)$ we have
\begin{align*} \label{const3-1}
        & \S(\nu) \\ 
       = & \sum_{i=1}^{a(k)} \frac{m^2s(k)^4}{n^2}\S(\nu(i)) + 
        \frac{(n-ma(k)s(k)^2)^2}{n^2}\S(\nu(0))+ \frac{1}{n^2}\sum_{w_i,w_j 
                \text{ not in same } \L_l} \log |w_i - w_j|. \end{align*}
Recall that $\rho(k) < \frac{1}{k} \cdot \min_{i\ne j} |z_i - z_j|$. 
This implies that for $w_i,w_j$ from $\L_\alpha,\L_\beta$ respectively with 
$\alpha  \ne \beta
\ne 0$, we have  \[(1-2/k)|z_\alpha -z_\beta|  \le |w_i-w_j| 
\le (1+2/k) |z_\alpha - z_\beta |.\] 
On the other hand, if $w_i \in \L_0$ and $w_j \notin \L_0$, then 
\[0\le \log|w_i-w_j| \le \log |1+|\I|+D+\rho(k)|,\] 
where $D$ is the diameter of the support of $\mu_k$. Hence  
\begin{align*} & \l|\frac{1}{n^2}\sum_{\text{ Exactly one of } w_i,w_j \in \L_0} 
\log |w_i - w_j| \r| \\  \le & \frac{a(k)s(k)^2(n-ma(k)s(k)^2)}{n^2}\log |1+|\I|+D+\rho(k)| 
\le C(k)/n^2\end{align*} for some function $C(k)$.  
Hence we have \begin{align*} 
        & \log |1-2/k| -\frac{C(k)}{n^2} \\ \le & \l|\frac{1}{n^2}\sum_{w_i,w_j 
                \text{ not in same } \L_l} \log |w_i - w_j|  -  
                \S(\mu_k) \r| \\ \le & \log|1+ 2/k| + \frac{C(k)}{n^2} . 
        \end{align*} This is true for every $\nu \in \n_n(\eps,k)$, 
        and therefore for every $\nu \in \wn_n(\eps,k)$.
By definition of $\wn_n$, we have $\S(\nu(i))\ge 2\G(U)\ge 2B(k)$ 
for each $\nu \in \wn_n$ and each $0 \le i \le a(k)$, where the function 
$B(k)=\min (\G(U),\G(I))$. Observe that, for a fixed interval $I$, the quantity $B(k)=\G(U)$ for large enough $k$, and therefore is equal to $\log \rho(k) (1+ o(1))$.

We thus have  
\[\sum_{i=1}^{a(k)} \frac{m^2s(k)^4}{n^2}
\S(\nu(i)) + \frac{(n-ms(k)^2)^2}{n^2}
\S(\nu(0)) \ge \frac{\log \rho(k) (1+o(1))}{a(k)} \ge \frac{2 \log \rho(k)}{a(k)}. \]  This completes the proof of part (ii) of the proposition.

Part (iii) of the proposition is immediate as
the statement holds for
all measures $\nu \in \n_n(\eps,k)$
and $\wn_n(\eps,k) \subset \n_n(\eps,k)$.
\end{proof}

\section{Proof of Proposition \ref{stage1-1}}
\label{sec-proofstage1-1}
We begin with a general approximation result.
\begin{lemma}
 \label{sampling}
 Let $\mu \in \mm^{\mathrm{Sym}}(\C)$ be of compact support and
such that 
$\Sigma(\mu)<\infty$ and $\mu$ has a bounded density 
with respect to the Lebesgue measure. Then there exists a
sequence of point configurations with distinct points and
empirical measures 
$\nu_k$ such that
\begin{itemize}
	\item[(1)] $\nu_{k}\in\mm^{\mathrm Sym}(\C)$,
                the support of $\nu_{k}$ is contained inside the 
                $1/a(k)$ thickening of the support of $\mu$, 
		and $\nu_k$ does not charge the real line.
        \item[(2)] $\nu_{k} \to \mu$ as $k \to \infty$.
        \item[(3)] $\Sigma_{\mathrm{a}}(\nu_{k}) \to \Sigma(\mu) $ 
                as $k \to \infty$.
        \item[(4)] The minimal separation between two distinct atoms of $\nu_k$ is at least $a(k)^{-2}$, where $a(k)$ is the 
                  number of atoms of $\nu_k$.
\end{itemize}
\end{lemma}
\begin{proof}
For an discrete measure $\nu$ having $n$ distinct atoms and 
equal mass $1/n$ on each atom, and a function $f$, 
define $\Sigma_{\mathrm{a}}^f(\nu)=\frac{1}{n^2}\sum_{i \ne j} 
f(x_i,x_j)$ where $\{x_i\}_{i=1}^n$ are the atoms of $\nu$. 
Thus, for $f(z,w)=\log|z-w|$, we have $\Sigma_{\mathrm{a}}^f(\nu)
=\Sigma_{\mathrm{a}}(\nu)$ as defined earlier.
 
Recall that the measure $\mu$ is compactly supported and 
symmetric under conjugation. Let $\mu_R$ be $\mu$ restricted to $\R$ and let $\mu_u$ be $\mu$ restricted to the upper half plane.  Define the measures $\mu_1=\frac{1}{2}\mu_R + \mu_u$ and $\mu_2=\frac{1}{2}\mu_R + \ol{\mu_u}$ where $\ol{\mu_u}$ is the measure supported on the lower half plane and defined by $\ol{\mu_u}(A)=\mu_u({\ol{A}})$. Obviously, $\mu=\mu_1 + \mu_2$.
 
For each $n$, we will obtain conjugation symmetric 
point sets of size $2n$ whose empirical measures will approximate 
$\mu$ in the following way. First, consider $n$ i.i.d. random samples  
$\{X_1,\cdots,X_n\}$ from $\mu_1$ (to obtain random samples we
consider the probability measure obtained by appropriately 
normalizing $\mu_1$). Consider the point set $Y:=\{Y_1,\cdots,Y_n\}$ 
where $Y_j=X_j + \frac{i}{n}$ and $i$ is the imaginary unit. 
Consider the point set $Z:=Y \cup \ol{Y}$, and let $L_n=\frac{1}{2n}
\sum_{z\in Z} \delta_z$ denote the (random) empirical measure associated 
with $Z$; one has that $L_n\to\mu$ in distribution (for example, by 
an application of Sanov's theorem).

Fix a positive number $M$ (to be thought of as large).
Let $K>1$ be a bound on the diameter of the support of $\mu$.
Set $f(z,w)=\log|z-w|$. 
 Define $f_M(z,w)=f(z,w)\wedge M \vee (-M)$ and $g_M=f-f_M$.
Because  $\Sigma(\mu)<\infty$ we have that
\[\alpha(M)=\iint |g_M(z,w)|d\mu(z)d\mu(w) \to 0\; \text{as $M\to\infty$}.\]

We have that
$\frac{1}{n} \le |Y_i - \ol{Y}_i| \le K$.
Therefore, for $n>n_0(K)$,
\[\frac{1}{n^2} \l| \sum_{i=1}^n \log|Y_i - \ol{Y}_i| \r| 
\le  \log \,n/n.\] 
On the other hand, for $i \ne j$ we have $\E[|g_M(Y_i,Y_j)|]=\iint 
|g_M(z,w)|d\mu_1(z)d\mu_1(w)$ and $\E[|g_M(\ol{Y}_i,\ol{Y}_j)|]=
\iint |g_M(z,w)|d\mu_2(z)d\mu_2(w)$.
  
We next claim that, for $M>\log K$,
we have  $\E[|g_M(Y_i,\ol{Y}_j)|] \le \iint |g_M(z,w)|d\mu_1(z)d\mu_2(w)$. 
To see this, note that 
$$g_M(z,w) = \left\{\begin{array}{cl}
        0,& \mbox{\rm if} \, -M \le \log|z-w| \le M,\\
        \log|z-w| - M,& \mbox{\rm if}\,  \log|z-w| \ge M,\\
        \log|z-w| + M& \mbox{\rm  if}\, \log|z-w| \le -M.
\end{array}
\right.$$
The case $\log|z-w| \ge M$ does not arise when we consider $g_M(Y_j,\ol{Y}_j)$
because $M>\log K$. On the other hand,
\[ |Y_k - \ol{Y}_j| = |X_k -\ol{X}_j + \frac{2i}{n}| \ge |X_k -\ol{X}_j|\] 
because the $X_j$-s belong to the upper half plane.  
Thus, $\log|Y_k-\ol{Y}_j| \le -M$ implies 
$\log|X_k-\ol{X}_j| \le -M$ and on this event we have 
$\log|X_k - \ol{X}_j| +M \le \log|Y_k - \ol{Y}_j| + M$. 
Since these quantities are negative, this is equivalent to 
$|\log|Y_k - \ol{Y}_j| + M| \le |\log|X_k - \ol{X}_j| +M|$. 
This implies that \[\E[|g_M(Y_i,\ol{Y}_j)|] \le \E[|g_M(X_i,\ol{X}_j)|]=  
\iint |g_M(z,w)|d\mu_1(z)d\mu_2(w),\] as claimed.
  
We now have that
\begin{eqnarray*}
&&      \E[|\Sigma_{\mathrm{a}}^{g_M}(L_n)|] \le 
\frac{1}{4n^2}\sum_{i \ne j} \E [|g_M(Z_i,Z_j)|]\\
&=&\frac{1}{4n^2} \sum_{i\ne j} \E[|g_M(Y_i,Y_j)|] + 
\frac{1}{4n^2} \sum_{i\ne j} \E[|g_M(Y_i,\ol{Y}_j)|] \\
&&+ 
\frac{1}{4n^2} \sum_{i\ne j} \E[|g_M(\ol{Y}_i,\ol{Y}_j)|] + 
\frac{2}{4n^2} \sum_{i=1}^n \E[|g_M(Y_i,\ol{Y}_i)|]. 
\end{eqnarray*}
Thus, from the previous computations, we can choose $n_1=n_1(M)$
so that for all $n>n_1(M)$,
$$\E[|\Sigma_{\mathrm{a}}^{g_M}(L_n)|] \le 2\alpha(M)\,.$$
In particular, for fixed $\delta>0$ 
there exists $n_2(M)=n_2(M,\delta)$ so that for $n>n_2(M)$,
there exists a realization $\nu_n$ of $L_n$ such that 
$d(\nu_n,\mu)\leq \delta$ and
$|\Sigma_{\mathrm{a}}^{g_M}(\nu_n)| \le 2\alpha(M)$. 
Applying a diagonalization
argument (with $\delta\to 0$ while $M$ is kept fixed), we
find a sequence (with some abuse of notation, denoted $\nu_k$, which has $a(k)$
atoms), so that
such that $d(\nu_{k},\mu)\leq 1/k$ and  
$|\Sigma_{\mathrm{a}}^{g_M}(\nu_{k})| \le 2\alpha(M)$. 
 
Now, $\Sigma_{\mathrm{a}}(\nu_{k})=
\Sigma_{\mathrm{a}}^{f_M}(\nu_{k})+\Sigma_{\mathrm{a}}^{g_M}(\nu_{k})$. 
Since $f_M$ is bounded and continuous, we have
$\Sigma_{\mathrm{a}}^{f_M}(\nu_{k}) \to  \iint f_M(z,w)d\mu(z)d\mu(w)$ 
as $k \to \infty$. This implies that  
\[\varlimsup_{k \to \infty} |\Sigma_{\mathrm{a}}^f(\nu_{k}) -  
\iint f(z,w)d\mu(z)d\mu(w)| \le 3 \alpha(M).\] 
Applying again a diagonalization argument (this time over $M$), one obtains 
the desired convergence.  

It only remains to show that the minimal separation between distinct atoms of $\nu_k$
is at least $a(k)^{-2}$. To see this, we observe that the restriction of $\nu_k$ is a 
realization of $a(k)/2$ i.i.d. samples from $\mu$, after a translation by $i/a(k)$ which does not change their mutual separation.
Also, this translation ensures that the separation between two atoms in the upper and lower half planes is at least $a(k)^{-1}$.
For atoms in the upper half plane (a fortiori in the lower half plane) we proceed as follows. Since $\mu$ has a bounded density with 
respect to the Lebesgue measure, therefore the probability that two independent samples from $\mu$ are at a mutual distance 
$< a(k)^{-2}$ is $C a(k)^{-4}$ for some universal constant $C_1$. If it is not the case that the minimal separation between the $a(k)/2$ atoms in the upper half plane 
is $\ge a(k)^{-2}$, then at least one of the ${ a(k)/2 \choose 2 }$ pairs of atoms must violate this condition. By a union bound, 
the probability of this is at most $C_2a(k)^{-2}$ for some universal constant $C_2$. Since $a(k)\ge k$, we have $\sum_k a(k)^{-2}<\infty$. By the Borel Cantelli lemma, 
this implies that $\nu_k$ has the desired property, eventually for large enough $k$.

\end{proof}

\begin{proof}[Proof of Proposition \ref{stage1-1}.]
Let $\nu_k$ be the sequence of atomic measures constructed 
in Lemma \ref{sampling}. Note that each on the $\nu_k$s  has $a(k)$ 
distinct 
(non real) atoms, is supported within the $1/a(k)$ thickening of
the support of $\mu$, and satisfies $\S(\nu_k)\to
\Sigma(\nu_k)$. Each $\nu_k$ gives rise to a monic polynomial
$Q_k$ with distinct zeros 
(so that the zeros of $Q_k$ are precisely the atoms of 
$\nu_k$).  

Observe that the measure $\mu$ is in the class of measures $\wt{\mm}(K)$ (with $K$ being a slight thickening of the support of $\mu$ so that $K$ also includes the supports of the measures $\nu_k$ for large enough $k$). 
As a result,  Theorem \ref{convergence1} implies that $\nu_k\in\wt{\mm}(K)$,
for $k$ large enough. As a result, we have $|\LL_{\nu_k}(z)| < 
\LL_{\nu_k}(|z|) $ (for all $z\not\in \R_+$) as soon as $k$ is large enough. If we write the polynomial $Q_k$  as 
\[Q_k(z)= a_0 + a_1 z + \cdots + a_{d-1}z^{d-1}+ z^d,\]
then the inequality of potentials turns into the inequality $|Q_k(z)|<Q_k(|z|)$.
Moreover, due to conjugation-symmetry of the roots of $Q_k$ (all of which are non-real) we deduce that $a_0$ is positive, whereas the fact that $a_1^{\nu_k}>0$ (refer to Remark \ref{coeff}) implies that $a_1>0$. Similar considerations on the potential, after the transformation of potentials $u(z) \mapsto \log|z| + u(\frac{1}{z})$ (equivalently $z \mapsto \frac{1}{z}$ at the level of measures), ensures that the coefficient $a_{d-1}$ is also positive. Now we can invoke De Angelis' Theorem (\cite{DA}, quoted as Theorem \ref{dat} in the Appendix and as Theorem A in \cite{BES})
to show that for some $h(k)$ large enough, 
the polynomial $P_k=Q_k^{h(k)}$ has all coefficients 
real and positive. 

Note that the empirical measure of zeros of $P_k$ coincides
with the empirical measure of zeros of $Q_k$. Finally, the condition $U(\nu_k, \frac{1}{2}a(k)^{-2}) \subset \p$ follows from an application of Theorem \ref{convergence1} (in particular, refer to Remark \ref{impact}). This completes the proof.
\end{proof}

\section{The Bergweiler-Eremenko
        approximation and proof of
Proposition \ref{prop-era}.}
\label{erapprox}
We begin with a fixed measure $\mu \in \p$ satisfying $I(\mu)<\infty$.
Let $$ \widehat{\LL}_\mu(z)=\int_{|w|\leq 1} \log(|z-w|)d\mu(w)+
\int_{|w|>1} \log(|1-\frac{z}{w}|)d\mu(w)\,,$$
and denote by
$u(z)=\LL_{\mu}(z)=\int \log |z-w| d\mu(w)$ its 
logarithmic potential.
Because $I(\mu)<\infty$, the Bergweiler-Eremenko condition can
be written as
$u(z)\leq u(|z|)$ 
for all $z\in \C$. 

The Bergweiler-Eremenko approximation proceeds 
in five steps to construct
a sequence of approximations $\mu_i=\mu_i^\eps$,
$i=1,\ldots,5$, with $\mu_5\in\d$. 
In the $i$-th step, one starts with a measure $\mu_{i-1}$ 
(with $\mu_0=\mu$ and $u_0=u$) and constructs measures $\mu_{i} \in \p$ 
(depending on a small parameter $\eps>0$) such that 
$\mu_{i} \to \mu_{i-1}$ weakly as $\eps \to 0$. 
The measures $\mu_i$ are
defined via  subharmonic functions $u_i$, 
such that $\mu_i=(2\pi)^{-1}\Delta u_i$ in the sense of distributions.
One shows, see
Section 2 of \cite{BES}) that in each of the 5 steps, 
one has $L_{\mu_i}(z)=u_i(z)+ k_i^{\eps}$ where $k_i^{\eps}$ is a 
constant (as a function of $z$) depending on $\eps$, and
that $u_i(z)\to u_{i-1}(z)$ for each $z\in \R_+$, while in some of
the steps the above convergence will occur pointwise in $\C$. 
It is a consequence of Proposition 
\ref{convergence} below that in each of the Steps 1-5, 
$k_i^{\eps} \to 0$ as $\eps \to 0$. 
\subsection{Preliminaries}
We begin with several preliminary properties concerning the convergence of
subharmonic functions.

\begin{proposition}
\label{convergence}
Let $u_{\eps}$ be a sequence of subharmonic functions 
converging in the sense of distributions to a 
subharmonic function $u_0$ as $\epsilon\to 0$.
Assume further that $u_\eps$ converges pointwise to $u_0$ on $\R_+$.
Let $\mu_\eps, \eps\geq 0$ 
denote the Riesz measure of $u_\eps$ and assume that
$u_\eps(z)=\LL_{\mu_\eps}(z)+k_\eps$ where $k_\eps$ is independent of $z$.
Then $\mu_\eps\to \mu_0$ weakly and $k_\eps \to k_0$.
%
\end{proposition}

\begin{proof}
By standard results, see \cite[Appendix]{BES}, one gets
the weak convergence $\mu_\eps\to\mu_0$ and the a.e. (with respect 
to Lebesgue measure on $\R_+$) convergence of $\LL_{\mu_\eps}$ to $\LL_{\mu_0}$.
In particular, the convergence $L_{\mu_\eps}(z)\to L_{\mu_0}(z)$ occurs at
a point  $z\in \R_+$. This yields the convergence $k_\eps\to k_0$.
\end{proof}

We next show that pointwise monotone convergence of subharmonic functions 
implies the convergence of the associated logarithmic energies.
Results of this nature are known in the literature; here we 
follow an approach based on the proof of a related result in \cite{Doob}.
\begin{proposition}
\label{enconv}
Suppose $\{u_n\}$ is
a sequence of subharmonic functions decreasing pointwise to a
subharmonic function $u_0$, as $n \to \infty$. 
Let $\mu_n$ be the Riesz measure of $u_n$, and assume $u_n(z)=
\LL_{\mu_n}(z)+k_n$, with $k_n \to k_0$ as $n \to \infty$. Then 
$\Sigma(\mu_n)\to \Sigma(\mu_0)$ as $n \to \infty$.
\end{proposition}

\begin{proof}
By subtracting off $k_0$, we can assume $u_0=\LL_{\mu_0}$ and 
$k_{\eps} \to 0$. Recall
the notation $\Sigma(\mu,\nu)= \iint \log|z-w|d\mu(z)d\nu(w)$. 

  The hypotheses imply in particular that $\mu_n\to \mu_0$ weakly.
  The lower semicontinuity of $-\Sigma(\cdot)$ then implies
  that $\limsup \Sigma(\mu_n)\leq \Sigma(\mu_0)$. To see the other 
  direction, note that if either $n=0$ or $n>m$ we have
  \begin{eqnarray*}
    \Sigma(\mu_n)&=& \int \l(u_n(z) - k_n\r) d\mu_n(z) 
    = \int u_n(z) d\mu_n(z) - k_n \\
    &\le &
    \int u_m(z) d\mu_n(z) - k_n  =  \Sigma(\mu_m,\mu_n) +k_m -k_n  
    \le \Sigma(\mu_m,\mu_m) + 2(k_m-k_n)\\
    &=& \Sigma(\mu_m)+2(k_m - k_n), 
  \end{eqnarray*}
  where the monotonicity of the sequence $\{u_n\}$ was used in the 
  inequalities. 
  We conclude that $\liminf_{n\to\infty} \Sigma(\mu_n)\geq \Sigma(\mu_0)$, completing
  the proof.
\end{proof}

\subsection{Approximation steps}
We describe each of the approximation steps $\mu_i^\epsilon\to\mu_{i-1}$,
$i=1,\ldots,5$,
and show that for each, both
\begin{equation}
  \label{eq-goal}
  \LL_{\mu^\epsilon_i}(1)\to \LL_{\mu_{i-1}}(1)\,,\quad
\Sigma(\mu^\epsilon_i)\to \Sigma(\mu_{i-1}).
\end{equation}
In the sequel, we  omit the subscript $\epsilon$ when it is clear from 
the context.

\subsubsection{Step 1:}
Given $\eps>0$, 
define \[u_1(z)=\max\{ u(ze^{i\alpha}): |\alpha| \le \eps  \}.\] 
Note that $u_1\geq u$ pointwise, and that $u_1$ decreases in $\epsilon$.
It is proved in \cite[Section 2]{BES} that $u_1\to_{\eps\to 0}
u$ weakly; however, due to upper semicontinuity of $u$, this implies
also the pointwise convergence $u_1\searrow_{\eps\to 0} u$. 
An application of
Propositions \ref{convergence} and \ref{enconv} yields
\eqref{eq-goal} for $i=1$.

\subsubsection{Step 2:} For $\eps_1 \in (0,\eps)$ where $\eps$ 
is as chosen in Step 1, define
$D_{\eps_1}=\{ z: |\arg(z)|\le \eps_1\}$. 
Let $v$ denote the solution to the Dirichlet problem in $D_{\eps_1}$ 
with boundary conditions $u_1(z)$ and $v(z)=O(\log |z|)$ as $z \to \infty$. 
Define $u_2=u_2(\eps_1)$ to be the function obtained by 
``balayage'' (i.e., sweeping out of the Riesz measure) from the domain 
$D_{\eps_1}$. In other words, $u_2(z)=v(z)$ if 
$z \in D_{\eps_1}$, and $=u_1(z)$ otherwise. 
In this step too, 
it follows from \cite[Section 2]{BES} and the upper semicontinuity
of $u_1$ that $u_2\searrow_{\eps_1\to 0}
u_1$ pointwise, and that $u_2(z)<u_2(|z|)$ for $z\in\C\setminus \R_+$. 
An application of
Propositions \ref{convergence} and \ref{enconv} yields
\eqref{eq-goal} for $i=2$.

\subsubsection{Step 3:} For $\eps >0$, define $u_3(z)=u_2(z+\eps)$. 
We have $d\mu_3(z)=d\mu_2(z+\eps)$. We have that $u_3\to u_2$ pointwise on
$\R_+$ and hence Proposition \ref{convergence} yields
that $\LL_{\mu_3}(1)\to \LL_{\mu_2}(1)$. Since
$\Sigma(\mu_3)=\Sigma(\mu_2)$, we conclude that \eqref{eq-goal} holds for 
$i=3$.

\subsubsection{Step 4:}
For $\eps>0$, define $v(z)=u_3(1/z)+\log|z|$ for 
$z \ne 0$, and extend $v(\cdot)$ to $\C$ by defining 
$v(0)=\limsup_{z\to 0} v(z)$. This definition preserves the
sub-harmonicity of $v$, and in fact 
\begin{equation}
  \label{eq-v0}
  v(0)=\lim_{r\searrow 0} v(r)
\end{equation}
because $u_3(z)\leq u_3(|z|)$, hence the limsup is
attained as $r\searrow 0$, and the convexity of $v(r)$ in $\log r$ then yields
the existence of the limit. We claim that in fact,
$v(0)>-\infty$ (and therefore, by sub-harmonicity, is finite). Indeed, 
$\LL_{\mu_v}(0)=-\LL_{\mu_3}(0)$; by construction, $u_3$ is harmonic
in a neighborhood of $0$, see \cite{BES}. Hence, $\LL_{\mu_v}(0)=-\LL_{\mu_3}(0)$ is finite, as claimed.
We note in passing that $\LL_{\mu_v}(1)=\LL_{\mu_3}(1)$.

Next,
define $w(z)=v(z+\eps)$ and, finally, 
$u_4(z)=w(1/z)+\log |z|$.  As in the previous step, we 
have that $\LL_{\mu_4}(1)=\LL_{\mu_w}(1)\to_{\eps\to 0} \LL_{\mu_v}(1)$. So
it only remains to check the convergence of the
logarithmic energy. To that end, note that under the transformation
$v(z)=u(1/z)+\log|z|$, one has
\begin{eqnarray*}
  \Sigma(\mu_v)&=&
  \iint \log|z-w|d\mu_v(z)d\mu_v(w)=
  \iint \log\left|\frac1z-\frac1w\right|d\mu_u(z)d\mu_u(w)\\
  &=&
\iint  \log |z-w| d\mu_u(z) d\mu_u(w) -2 \int \log|z| d\mu_u(z) =
\Sigma(\mu_u)-2 \LL_{\mu_u}(0) . \end{eqnarray*}
With this computation in hand, we trace the changes in the logarithmic energy 
$\Sigma$ in Step 4 as follows:
\begin{eqnarray*}
  \Sigma(\mu_v)&=&\Sigma(\mu_{u_3}) -2 \LL_{\mu_3},
 \Sigma(\mu_w)=\Sigma(\mu_v), \LL_{\mu_w}(0)=\LL_{\mu_v}(\eps)\\
 \Sigma(\mu_{u_4})&=&\Sigma(\mu_w)-2 \LL_{\mu_w}(0) = 
 \Sigma(\mu_v)-2 \LL_{\mu_v}(\eps) = \Sigma(\mu_{u_3}) -2 \LL_{\mu_{u_3}}(0) -2 \LL_{\mu_v}(\eps) \\
&=&\Sigma(\mu_{u_3}) +2 \LL_{\mu_v}(0) -2 \LL_{\mu_v}(\eps). 
\end{eqnarray*}
But \[ \LL_{\mu_v}(\eps) -  \LL_{\mu_v}(0)=v(\eps) - v(0) \to 0\,,\]
  due to \eqref{eq-v0}. This completes the proof of 
\eqref{eq-goal} for Step 4.

\subsubsection{Step 5:}  
In this step, we slightly differ from the recipe of \cite{BES}.
Let $I$ denote an interval of length $\t$ (which is assumed to be small 
but fixed) centered at -1. Let $\alpha$ be the normalized 
Lebesgue measure supported on the set $I$. 
For $\eps>0$, define the measure \[\mu_5:=(1-\eps)\mu_4  + \eps \alpha.\] 
Define $u_5= \LL_{\mu_5}$. 
Via a series expansion of $\log|1+z|$ for small and large $z$, one
checks that this has the intended effect of making the coefficients 
$b$ and $c$ in Step 5 of the \cite{BES} argument positive, 
while preserving 
the inequality $u_5(z) < u_5(|z|)$ for $z\not\in \bar\R_+$. 
Thus, $\mu_5\in \d$.

To see \eqref{eq-goal} for $\mu_5$, note that
$\LL_{\mu_5}= (1-\eps)\LL_{\mu_4}  + \eps \LL_{\alpha}$ 
and \[\Sigma(\mu_5)= (1-\eps)^2\Sigma(\mu_4)+ \eps^2 \Sigma(\alpha) + 
  2\eps(1-\eps)\int \LL_{\alpha}(z)d\mu_4(z).\] Since $I$ is a bounded 
  interval, one has that
  $\LL_{\alpha}(z)$ is uniformly bounded in $z$ on any compact set, 
  in particular, on the support of $\mu_4$. 
  Letting  $\eps \to 0$, one gets \eqref{eq-goal} for $\mu_5$.
  
\subsubsection{Step 6:}  
We convolve $\mu_5$ with a measure $\eta_{\eps}$ (for small enough $\eps$) to obtain $\mu_6$. The measure $\eta_\eps$ 
is required to have a smooth density compactly supported on a ball of radius $\eps$. It can be easily seen that $\mu_6$
so defined has a bounded density with respect to the Lebesgue measure, given by $\psi(x)=\E_{\mu_5}[\eta_\eps(x+Z)]$.
By taking $\eps$ small enough we can ensure that $\mu_6 \in \p$, $\mu_6$ has the desired support properties, and the correct behaviour of the potential 
at 0 and $\i$, using Theorem \ref{convergence1}. $\log|1-z|$ is a 
bounded continuous function on the support of these measures, so \[\int \log|1-z| d(\mu_5 * \eta_\eps)(z) \to \int \log|1-z|d\mu_5(z). \] 
For the convergence of the logarithmic energies $\G(\mu_5 * \eta_\eps) 
\to \G(\mu_5)$, one argues as follows. Let $K$ be the thickening of the (compact)
support of $\mu_5$ by $1$. By the upper-semicontinuity 
of the logarithmic energy in $\M_1(K)$, we get that 
$\limsup_{\eps\to 0} \G(\mu_5*\eta_\eps)\leq \G(\mu_5)$. To get the reverse 
inequality, let $W$ be distributed according to 
$\eta$. Then, for any $z\in K$ and $\eps<1$,
$$\E\log|z+\eps W|\geq \log |z|+\E\log|1+\eps W/|z||\geq \log |z|- C_\epsilon\,,$$
where $C_\epsilon>0$ is independent 
of $z\in K$ and $\lim_{\eps\to0} C_\epsilon=0$. Noting that
$$\G(\mu_5*\eta_\eps)=\iint \E \log|x-y+\eps W| \mu_5(dx)\mu_5(dy)\,,$$
it follows that 
$\liminf_{\eps\to 0} \G(\mu_5*\eta_\eps)\geq \G(\mu_5)$, as claimed.
(This argument is similar to that used in the
proof of  \cite[Lemma 2.2]{BAZ}.)

  This completes the proof of Proposition 
  \ref{prop-era}. \qed

  \section{Conditioning on all zeros being  real}
  \label{sec-6}
 One notes from the expression for the density 
 \eqref{dist2} in case $k=0$ that $P(L_n\in \M_1(\R_-))>0$. One also
 notes that $\{\mu\in \p: \supp(\mu)=\R_-\}=\M_1(\R_-)$. Thus, one can 
rerun the proof of the lower bound in Theorem  
\ref{ldp} replacing throughout $\p$ by $\M_1(\R_-)$ as a particular case of
the proof in \cite{BAG}. One obtains that
$$ \lim_{n\to\infty} \frac1{n^2} \log \P_n(L_n\in \M_1(\R_-))=-I_R\,,$$
and one immediately deduces Theorem \ref{theo-Lihom} by noting that 
the minimizer $\mu_R$ is unique due to the strict convexity of $I$
(applied on $\M_1(\R_-)$).

To see Theorem \ref{theo-baik}, we can make the transformation $x\mapsto -x$
to see that we are interested in solving the variational problem
\begin{equation}
  \label{eq-baik1}
  \inf_{\mu\in M_1([1,\infty))}\{\int_0^\infty \log (x+1) d\mu(x)-
  \gamma \int_0^\infty \int_0^\infty \log|x-y|d\mu(x)d\mu(y)\,,
\end{equation}
  with $\gamma=1/2$.

  A standard application of calculus of variation methods (as e.g. in
  \cite[Lemma 2.6.2]{AGZ}) shows that the minimizer $\bar \mu$ of 
  \eqref{eq-baik1} is characterized as the unique solution, for some 
  constant $C$, of
  \begin{equation}
    \label{eq-baik3} 
    2\gamma \LL_{\bar \mu}(x)\,\left\{
    \begin{array}{ll}
      =\log (x+1)+C,& \bar \mu -a.e.\, \\
      >\log |x+1|+C,& x\in \R_+\setminus \supp(\bar \mu).
    \end{array}\right.
  \end{equation}
    One can proceed by first guessing the form of the 
    minimizer and then verifying that it satisfies indeed 
    \eqref{eq-baik3}. For $\gamma>1/2$, this is
    can be achieved solving, in a 
    compact interval, the associated Riemann-Hilbert problem, and then
    taking the  limit 
    $\gamma\to 1/2$. We do not detail these computations, instead 
    presenting the \textit{ansatz} that the minimizer in \eqref{eq-baik1}
    has density with respect to Lebesgue measure on $[0,\infty)$ of the form 
      \begin{equation}
        \label{eq-baik4}
        \psi(x)=\frac{1}{\pi (x+1) \sqrt{x}}\,.
      \end{equation}
      We need to verify that $\psi(x)dx$ satisfies \eqref{eq-baik3}. 
      Making the change of variables $w=\sqrt{x}$, we have
      $$\LL_{\bar \mu}(x)=\frac{2}{\pi}
      \int_0^\infty \frac{\log|x-w^2|}{w^2+1} dw\,.$$
      Choosing the contour of integration ${\cal C}:=\{r\}_{r=-R}^R
      \cup \{Re^{i\theta}\}_{\theta=0}^{\pi}$ for $R$ large,
      and noting the pole at $i$, one obtains from a residue computation that
      $\LL_{\bar \mu}(x)= \log(x+1)$ for $x\in \R_+$, i.e. that
      \eqref{eq-baik3}
      holds with density $\psi$. This completes the proof of Theorem 
      \ref{theo-baik}.

\section{Appendix: De Angelis' Theorem}
Here we state De Angelis' Theorem (\cite{DA}, also quoted as Theorem A in \cite{BES}) on the positivity of coefficients of polynomials, which has been used in the proof of our main theorem. 

\begin{theorem}[De Angelis]
\label{dat}
 Let $f(z)=a_0+a_1z + \cdots + a_d z^d$ be a real polynomial, with $a_0>0,a_d>0$. Then the following conditions are equivalent:
 \begin{itemize}
  \item[] (i) There exists a positive integer $m$ such that all coefficients of $f^m$ are strictly positive.
  \item[] (ii) There exists a positive integer $m_0$ such that for all $m \ge m_0$, all coefficients of $f^m$ are strictly positive.
  \item[] (iii) The inequalities \[ f(z) < f(|z|) \text{ for all } z \notin [0,\i) \] and \[ a_1>0, a_{d-1}>0 \] hold.
 \end{itemize}
\end{theorem}


\begin{thebibliography}{99999}
  \bibitem[AGZ10]{AGZ} G. W. Anderson, A. Guionnet, O. Zeitouni, \emph{An 
    Introduction to Random Matrices}, Cambridge University press, Cambridge,
    2010.
\bibitem[BAG97]{BAG} 
G. Ben Arous, A.   Guionnet, 
Large deviations for Wigner's law and Voiculescu's non-commutative entropy.
\emph{Probab. Theory Related Fields} {\bf  108}  (1997),  pp.
517--542.
\bibitem[BAZ98]{BAZ} G. Ben Arous, O.  Zeitouni, 
Large deviations from the circular law.
\emph{ESAIM Probab. Statist.} {\bf   2}  (1998), pp. 123--134
\bibitem[BE14]{BES}
W. Bergweiler, A. Eremenko,  Distribution of zeros of polynomials with positive coefficients. 
   \emph{ Ann. Acad. Sci. Fenn.}, 40 (2015) 375--383.
\bibitem[Be08]{berman} R. J. Berman, Determinantal point processes and fermions on complex manifolds: large deviations and bosonization.
  arXiv:0812.4224.
\bibitem[BRS86]{BA}
A. T.  Bharucha-Reid, M. Sambandham.  \emph{Random Polynomials.}
Academic Press,  Orlando, FL,  1986. 
\bibitem[Bl11]{Bloom}
T. Bloom and N. Levenberg,
Pluripotential energy and large deviation (2011).
        arXiv:1110.6593.
      \bibitem[DZ98]{DZ} A. Dembo, O. Zeitouni,
        \emph{Large Deviations techniques and Applications}, 2nd ed.,
        Springer, New-York, 1998.
\bibitem[D84]{Doob}
J. L. Doob, \emph{Classical Potential Theory 
and Its Probabilistic Counterpart}. Springer, Berlin--New-York, 1984.

\bibitem[dA03]{DA}
V. De Angelis, \emph{Asymptotic expansions and positivity of coefficients for large powers of analytic functions}, Int. J. Math. Math. Sci., 16 (2003),
 1003--1025.

\bibitem[ET50]{erdos-turan}
  P. Erd\"{o}s,  P. Tur\'{a}n, P,  On the distribution of roots of
  polynomials.
  \emph{Ann.  Math.} {\bf 51},  (1950), pp. 105--119.
  \bibitem[Hor94]{Hor} L. H\"{o}rmander, \emph{Notions of convexity},
	  Birkhauser (1994).
\bibitem[IZ95]{IbragimovZ} I. Ibragimov,
  O. Zeitouni, 
On roots of random polynomials. 
\emph{Trans. Amer. Math. Soc.} {\bf 349} (1997), pp. 2427--2441. 
\bibitem[Li11]{Li} W. V. Li, 
Probability of all real zeros for random polynomial
with the exponential ensemble. Preprint (2011), 
available at 
http://citeseerx.ist.psu.edu/viewdoc/download?rep=rep1\&type=pdf\&doi=10.1.1.221.6122
\bibitem[MO13]{MO} MathOverflow thread:\\ 
        http://mathoverflow.net/questions/134998/zeros-of-polynomials-with-real-positive-coefficients.
\bibitem[Ob23]{obrechkoff} N. Obrechkoff, Sur un probl\`{e}me de Laguerre,
  \emph{Comptes Rendus Acad. Sci. Par.} {\bf 177} (1923), pp. 102--104.
\bibitem[R95]{Ra}
T. Ransford, \emph{Potential Theory in the Complex Plane}, 
London Mathematical Society, Student Texts 28, London 1995.
\bibitem[SS12]{Ser} E. Sandier, S. Serfaty,
  2D Coulomb Gases and the Renormalized Energy.
  arXiv:1201.3503.
\bibitem[SV95]{shepp-vanderbei}
L. A. Shepp, R. J.  Vanderbei,  The complex zeros of random polynomials.
\emph{Trans. Amer. Math. Soc.} {\bf   347}  (1995),  pp.  4365--4384.
\bibitem[SS62]{sparo-sur} D. I. \v{S}paro and S. M. \v{S}ur,
  {On the distribution of roots of random polynomials}, 
  \emph{Vestnik Moskov. Univ. Ser. I Mat. Meh.} {\bf 3} (1962), pp. 40--43
\bibitem[TV13]{TV}
T. Tao, V. Vu, 
\emph{Local universality of zeroes of random polynomials},
arXiv:1307.4357.
\bibitem[Za04]{Za}
D. N. Zaporozhets, 
On the distribution of the number of real roots of a random polynomial.
\emph{Zap. Nauchn. Sem. S.-Peterburg. Otdel. Mat. Inst. Steklov. (POMI)}
{\bf 320}  (2004),  Veroyatn. i Stat. 8, pp. 69--79. 
Translation in  \emph{J. Math. Sci. (N. Y.)}
{\bf 137}  (2006),  pp.  4525--4530.

\bibitem[ZZ10]{ZelditchZ} O.  Zeitouni, S. Zelditch, 
Large deviations of empirical measures of zeros of random polynomials. 
\emph{Int. Math. Res. Not. IMRN} {\bf 20} (2010), pp.   3935--3992. 
        

\end{thebibliography}
\end{document}